\documentclass[10pt,a4paper]{amsart}

\usepackage[utf8]{inputenc}
\usepackage[T1]{fontenc} %
\usepackage{amscd}
\usepackage{amsmath}
\usepackage{amsthm}
\usepackage{amsfonts}
\usepackage{amssymb}
\usepackage{graphicx}
\usepackage[usenames,dvipsnames,svgnames,table]{xcolor}
\usepackage{color}
\usepackage{float}
\usepackage[all]{xy}

\usepackage[english]{babel}
\usepackage{color}
\usepackage{fancyhdr} 

\fancypagestyle{plain}{ %

 \fancyhf{} %

}

\newtheorem{theorem}{Theorem}[section]
\newtheorem{prop}[theorem]{Proposition}

\newtheorem{lem}[theorem]{Lemma}
\newtheorem{dfn}[theorem]{Definition}
\newtheorem*{dfn*}{Definition}

\newtheorem{conjecture}[theorem]{Conjecture}
\numberwithin{equation}{section}

\newtheorem{maintheorem}{MAIN THEOREM}

\def\RR{\mathbb{R}}

\def\Hnuno{\mathcal{H}^{n-1}}

\definecolor{grey}{rgb}{0.75,0.75,0.75}
\definecolor{orange}{rgb}{1.0,0.5,0.5}
\definecolor{brown}{rgb}{0.5,0.25,0.0}
\definecolor{pink}{rgb}{1.0,0.5,0.5}
\definecolor{green}{rgb}{0.0,0.533,0.133}
\definecolor{darkred}{rgb}{0.3,0.0,0.0}
\newcommand{\strong}[1]{\textbf{#1}}

\newcommand{\nin}{\not\in}

\newcommand{\tmop}[1]{\ensuremath{\operatorname{#1}}}

\newenvironment{itemizedot}{\begin{itemize}

	}
	{ \end{itemize}}
\newenvironment{itemizeminus}{\begin{itemize}

	}
	{\end{itemize}}

\newenvironment{remark}{\medskip\noindent{\strong{Remark.}}}
\newcommand{\Tone}{T_{p_1}M_1 }

\newcommand{\T}{T_p M}
\newcommand{\e}{\exp_p}
\newcommand{\cleaveq}{q=e_1(x_1)=e_1(x_2)}
\def\NC{\mathcal{NC}}
\newcommand{\SAtwo}{\mathcal{SA}_2}
\begin{document}
\title{Linking curves, sutured manifolds and the Ambrose conjecture for generic 3-manifolds}
\def\titulin{Linking curves, sutured manifolds and the Ambrose conjecture}

\author{Pablo Angulo Ardoy}
\address{ Department of Mathematics, Universidad Aut\'onoma de Madrid}
\email{pablo.angulo@uam.es}
\thanks{The author was partially supported by research grant ERC 301179, and by INEM}

\begin{abstract}
We present a new strategy for proving the {\emph{Ambrose conjecture}}, a global version of the Cartan local lemma.
We introduce the concepts of linking curves, unequivocal sets and sutured manifolds, and show that any sutured manifold satisfies the Ambrose conjecture.
We then prove that the set of sutured Riemannian manifolds contains a residual set of the metrics on a given smooth manifold of dimension $3$.
\end{abstract}

\maketitle

\section{Introduction}

Let ($M_{1} ,g_{1}$) and ($M_{2} ,g_{2}$) be two complete Riemannian manifolds
{\emph{of the same dimension}}, with selected points $p_{1} \in M_{1}$ and
$p_{2} \in M_{2}$. 
Any linear map $L:T_{p_{1}} M_{1} \rightarrow T_{p_{2}} M_{2}$ induces a map between the pointed manifolds ($M_{1} ,p_{1}$) and ($M_{2} ,p_{2}$): 
$\varphi = \exp_{p_2} \circ L \circ ( \exp_{p_1} |_{O_{1}})^{-1}$
defined in $\varphi(O_1)$, for any domain $O_{1} \subset T_{p_{1}} M_{1}$ such that
$e_1|_{O_{1}}$ is injective (for example, if $\exp_{p_1}O_{1}$ is a normal neighborhood of
$p_{1}$).

A classical theorem of E. Cartan {\cite{Cartan51}} identifies a situation where this map is an isometry.

For $x\in T_{p_{1}} M_{1}$, let $\gamma_{1}$ be the geodesic on $M_1$ defined in the interval $[0,1]$, starting at $p_{1}$ with initial speed vector $x$ and $\gamma_{2}$ be the geodesic on $M_2$ starting at $p_{2}$ with initial speed $L(x)$.

Let $P_{\gamma}:T_{p_i}M_i\rightarrow T_{\gamma_i(1)}M_i$ denote parallel transport along a curve $\gamma$.

\begin{dfn}\label{def: parallel translation of curvature}
  The curvature tensors of $(M_{1},p_{1} )$ and $(M_{2} ,p_{2} )$ are $L$-related if and only if for any $x\in T_{p_1}M_1$:

\begin{equation}\label{eqn: parallel transport of curvature along geodesics agrees}
P_{\gamma_1}^\ast(R_{\gamma_1(1)}) = L^\ast P_{\gamma_2}^\ast(R_{\gamma_2(1)}) 
\end{equation}
\end{dfn}

In the definition, $P_{\gamma_i}^\ast(R_{\gamma_i(1)})$ is the pull back of the $(0,4)$ curvature tensor at $\gamma_i(1)\in M_{i}$ by the linear isometry $P_{\gamma_i}$, for $i=1,2$
$$
P_{\gamma_i}^\ast(R_{\gamma_i(1)})(v_1,v_2,v_3,v_4) = R_{\gamma_i(1)}\big(P_{\gamma_i}(v_{1}) ,P_{\gamma_i}(v_{2}) , P_{\gamma_i}(v_{3}), P_{\gamma_i}(v_4)\big)
$$
for any four vectors $v_{1} ,v_{2} , v_{3}, v_{4}$ in $T_{p_{i}} M_{i}$, and $L^\ast$ is used to carry the tensor $P_{\gamma_2}^\ast(R_{\gamma_2(1)})$ from $p_2\in M_2$ to $p_1\in M_1$.

The usual way to express that the curvature tensors of $(M_{1}, p_{1} )$ and $(M_{2}, p_{2} )$ are $L$-related is to say that the {\emph{parallel translation of curvature along corresponding geodesics}} on $M_{1}$ and $M_{2}$ coincides.
This certainly holds if $L$ is the differential of a global isometry between $M_1$ and $M_2$.

\begin{theorem}\label{Cartan's theorem}
  If the curvature tensors of $(M_{1} ,p_{1} )$ and $(M_{2} ,p_{2} )$ are $L$-related, and $\exp_{1}|_{O_{1}}$ is injective for some domain $O_{1} \subset T_{p_{1}} M_{1}$, then $\varphi = \exp_{2} \circ L \circ ( \exp_{1} |_{O_{1}} )^{-1}$ is an isometric immersion.
\end{theorem}

\begin{proof}
  The proof of lemma 1.35 in \cite{Cheeger Ebin} works for any domain $O$ such that $\exp_{1}|_{O}$ is injective.
\end{proof}

In 1956 (see \cite{Ambrose}), W. Ambrose proved a global version of the above theorem, but with stronger hypothesis.

A \emph{broken geodesic} is the concatenation of a finite amount of geodesic segments.
The Ambrose's theorem states that if the parallel translation of curvature along \emph{broken geodesics} on $M_{1}$ and $M_{2}$ coincide, and both manifolds are \emph{simply connected},
then the above construction gives a global isometry $\varphi:M_1\rightarrow M_2$ whose differential at $p_1$ is $L$.
It is enough if the hypothesis holds for broken geodesics with only one ``{\emph{elbow}}'' (the reader can find more details in \cite{Cheeger Ebin}).

In \cite{Hicks}, in 1959, the result of W. Ambrose was generalized to parallel transport for affine connections; in \cite{Blumenthal Hebda}, in 1987, to Cartan connections; and in \cite{Pawel Reckziegel}, in 2002, to manifolds of different dimensions.

The Ambrose's theorem has found applications to inverse problems (see \cite{Lassas_etc} and \cite{Kurilev Lassas Uhlmann}).

Ambrose also posed the following conjecture:

\begin{conjecture}[\textsc{Ambrose Conjecture}]\label{ambrose conjecture original version}
  Let ($M_{1} ,p_{1}$) and ($M_{2} ,p_{2}$) be two simply-connected Riemannian manifolds with $L$-related curvature.

  Then there is a global isometry whose tangent at $p_{1}$ is $L$.
\end{conjecture}

Ambrose himself was able to prove the conjecture if all the data is analytic.
In 1987, in the paper \cite{Hebda}, James Hebda proved that the conjecture was true for surfaces that satisfy a certain regularity hypothesis, that he was able to prove true in 1994 in \cite{Hebda94}. J.I. Itoh also proved the regularity hypothesis independently in \cite{Itoh96}.
The latest advance came in 2010, after we had started our research on the Ambrose conjecture, when James Hebda proved in \cite{Hebda10} that the conjecture holds if $M_1$ is a \emph{heterogeneous manifold}. Such manifolds are generic.

The strategy of James Hebda in \cite{Hebda} can be rephrased in the following terms: in any Riemannian surface, for any cleave point $q$, there is always a \emph{cut locus linking curve} (see definition \ref{def: cut locus linking curve}) that joins the two minimizing geodesics that reach $q$.
We prove in Theorem \ref{thm: any manifold has a metric without cutLC} that this strategy does not carry over to higher dimension, and present a new strategy towards a proof for the Ambrose conjecture in dimension greater than $2$.

We refer the reader to section \ref{section: unequivocal y linked} for the terms used in the following definition:

\begin{dfn}\label{dfn: sutured manifold}
 A pointed manifold $(M_0,p_0)$ is \strong{sutured} (resp. \emph{strongly sutured}) if and only if for any $x\in T_{p_0}M_0$, there is an unequivocal $y\in T_{p_0}M_0$ with $\|y\|\leq\|x\|$ that is \emph{linked} to $x$ (resp. \emph{strongly linked}).
\end{dfn}

\begin{maintheorem} \label{main theorem ambrose}
 The Ambrose conjecture holds if $(M_{1} ,p_{1} )$ is a sutured manifold.
\end{maintheorem}

Our conjecture is that all manifolds are strongly sutured, but in this paper we only prove it for manifolds whose exponential map has some generic transversality properties (see definition \ref{the set of generic metrics}):

\begin{maintheorem} \label{main theorem generic}
 The set of \emph{strongly sutured} Riemannian metrics on a $3$ dimensional differentiable manifold $M$ contains a residual set of metrics.
\end{maintheorem}

The proof of Main Theorem \ref{main theorem generic} involves several technical difficulties but it is also quite constructive: we build linking curves using the \emph{linking curve algorithm} (although there is a non-deterministic step at which we have to choose one curve that avoids some obstacles).
In theorem \ref{thm: the linking curve algorithm stops in finite time}, we prove that the algorithm always produces a special type of linking curve, starting on any conjugate point.

The proof of the Ambrose conjecture in \cite{Hebda10} certainly works for a generic class of Riemannian manifolds of any dimension, and is shorter than the proof presented here.
However, that proof does not seem to be extendable to arbitrary metrics.
Indeed, the class of Riemannian manifolds for which we prove the Ambrose conjecture is not contained in the corresponding class in \cite{Hebda10}, so this is truly a different approach.
In the last section, we show how the ideas in this paper could be used to complete the proof of the conjecture for all Riemannian manifolds.

For the proof of these results we have introduced some new concepts that we believe are interesting in their own sake, such as \emph{linking curves} and \emph{synthesis manifolds} in section \ref{section: synthesis} or \emph{conjugate descending flow} in section \ref{section: conjugate flow}.

The outline of the paper is as follows:
In section \ref{section: James Hebda's tree formed curves} we interpret the proof in \cite{Hebda82} in our own terms and show why it only works in dimension 2.
In section \ref{section: Affine development and tree formed curves} we study tree-formed curves and prove lemma \ref{Dev_2^-1 o L o Dev_1(u) = v} about the affine development of curves in manifolds with $L$-related curvature.
In section \ref{section: synthesis} we define quasi-continuous linking curves, unequivocal sets and synthesis manifolds, and prove our Main Theorem \ref{main theorem ambrose}.
In section \ref{section: generic metric} we collect useful results about the exponential map for a generic metric.
In section \ref{section: proof of main theorem generic} we define conjugate descending curves, prove that they are \emph{unbeatable}, define finite conjugate linking curves (FCLCs), and prove that they can be built for a generic metric using the \emph{linking curve algorithm}.

The results in this paper are mostly included in the author's thesis \cite{Angulo Tesis}, but they have been reorganized to make it more clear and more general, and a few short but powerful results have been added.
We warn the reader of that document that some definitions have changed with respect to that document.

\subsection{Acknowledgments}
We thank Yanyan Li and Biao Yin, who introduced the author to the Ambrose conjecture.
We also thank Luis Guijarro and James Hebda for their support and suggestions.

\section{Notation}\label{section: Ambrose notation}

$M$ is an arbitrary Riemannian manifold, $p$ a point of $M$, $(M_{1} ,p_{1} )$
and $(M_{2} ,p_{2} )$ are two Riemannian manifolds with $L$-related curvature.

Let $ e_{1}$ stand for $\exp_{p_{1}}$ and $ e_{2}$ for $\exp_{p_{2}} \circ L$.

We denote by $\tmop{Cut}_{p}$, the {\strong{cut locus}} of $M$ with respect to $p$ (see chapter 5 of \cite{Cheeger Ebin} for definitions and basic properties).
Let us define also the {\emph{injectivity set}} $O_{p} \subset \T$, consisting of those vectors $x$ in $\T$ such that $d( \e (tx),p)=t$ for all $0 \leqslant t\leqslant 1$, and let $\tmop{TCut}_{p} = \partial O_{p}$ be the {\emph{tangent cut locus}}.
The set $\tmop{TCut}_{p}$ maps onto $\tmop{Cut}_{p}$ by $\e$.

In our proof, we will make heavy use of a subset of $t_p M$ bigger than the injectivity set, defined as follows.
We define the functions $\lambda_{k} :S_{p_{1}} M_{1} \rightarrow \mathbb{R}$, where $\lambda_k(x)$ is the parameter $t_{\ast}$ for which $t_\ast \cdot x$ is the $k$-th conjugate point along $t \rightarrow \e(tx)$ (counting multiplicities).
These functions were shown to be Lipschitz in \cite{Itoh Tanaka 00}.
In \cite{Castelpietra Rifford}, it was shown that $\lambda_1$ is semiconcave.
Together with L.Guijarro, the author proved in \cite{Nosotros2} that these functions are also Lipschitz in Finsler manifolds.
We define $V_{1}$ as the set of tangent vectors $x$ such that $|x| \leqslant \lambda_{1} (x/|x|$), a set with Lipschitz boundary.
It is well known that $O_{p} \subset V_{1}$.

\section{James Hebda's tree formed curves}\label{section: James Hebda's tree formed curves}

\subsection{Tree formed curves}
Let $\tmop{AC}_{p} (X)$ be the space of absolutely continuous curves in the manifold $M$ starting at $p$, with the topology defined as in {\cite{Hebda}}.

Affine development $\tmop{Dev}_{p} : \tmop{AC}_{p} (M)\rightarrow \tmop{AC}_{0} (T_{p} M)$ for absolutely continuous curves is also defined in that reference, extending the common definition in \cite{Kobayashi Nomizu}.

\emph{Tree-formed} curves are similar to the {\emph{tree-like paths}} of the theory of rough paths (see \cite{Hambly Lyons}), but we will stick to the original definition in \cite{Hebda}.
The model for a tree-formed curve $u:[0,1] \rightarrow M$ is an absolutely continuous curve that factors through a finite topological tree $\Gamma$.
In other words, $u= \bar{u} \circ T$ for some continuous map $\bar{u}:\Gamma\rightarrow M$ and a map $T:[0,1]\rightarrow \Gamma$ with runs through each edge of the tree exactly twice, in opposite directions.
The tree $\Gamma$ is the topological quotient of the unit interval by the map $T$.

The definition also allows for a ``partial identification''.

\begin{dfn}\label{def: tree formed curves}
Let $T:[0,1] \rightarrow \Gamma$ be a quotient map, and $u$ an absolutely continuous curve.
Then $u$ is \strong{tree formed} with respect to $T$ if and only if
\[ \int_{t_{1}}^{t_{2}} \varphi (s)(u' (s))ds=0 \]
for any continuous $1$-form $\varphi$ along $u$ $( \varphi (s) \in
T^{\ast}_{u(s)} M)$ that factors through $\Gamma$ (this means that $T (s_{1} ) =T (s_{2} )$
implies $\varphi (s_{1} ) = \varphi (s_{2} )$), and for any $t_{1}$, $t_{2}$ such
that $T(t_{1} )=T(t_{2} )$.

If $T(0)=T(1)$, we say the curve is {\strong{fully tree-formed}}.
\end{dfn}

If $\Gamma =[0,1]$ and $T$ is the identity, the definition is empty, and we will rather use the definition saying
that a certain curve $u$ is tree-formed with respect to an identification map with $T(t_{1} )=T(t_{2} )$ as another way to say that $u|_{[t_{1} ,t_{2}]}$ is a fully tree-formed curve.

\begin{theorem}[{\cite[Theorem 3.3]{Hebda}}]\label{Hebda3.3}
Tree formedness is preserved by affine development:
\begin{itemize}
 \item If $u\in AC_p(M)$ is tree formed for an identification $T$, then $Dev_p(u)\in AC(T_pM)$ is also tree formed for the same $T$.
 \item If $v\in AC(T_pM)$ is tree formed for an identification $T$, then $Dev_p^{-1}(v)\in AC_p(M)$ is also tree formed for the same $T$.
\end{itemize}
\end{theorem}

\subsection{The proof of the Ambrose conjecture for surfaces by James Hebda}\label{subsection: Hebda's approach}

In this section we give a sketch of the paper {\cite{Hebda}}, which is important for later sections.
The reader can find more details in that paper.

Theorem \ref{Cartan's theorem} shows that $\varphi = \exp_{2} \circ L \circ (
\exp_{1} |_{U_{p_1}} )^{-1}$ is an isometric immersion from $U_{p_{1}} =M_{1}
\setminus \tmop{Cut}_{p_{1}}$ into $M_{2}$. The starting idea is to prove that
whenever a point in $\tmop{Cut}_{p_{1}}$ is reached by two geodesics
$\gamma_{1}$ and $\gamma_{2}$, meaning that $ e_{1} ( \gamma_{1}' (0))=  e_{1} (
\gamma_{2}' (0))$, then $ e_{2} ( \gamma_{1}' (0))=  e_{2} (
\gamma_{2}' (0))$. Then the formula $\varphi (p)=e_{2} (x)$, for any $x \in
(O_{p} \cup \tmop{TCut}_{p} ) \cap e_{1}^{-1} (p)$ gives a well-defined map
$\varphi :M_{1} \rightarrow M_{2}$ that is an isometry at least on
$U_{p_{1}}$.

It is a well-known fact that the cut locus looks specially simple at the {\emph{cleave points}}, for which there are exactly two minimizing geodesics from $p$, and both are non-conjugate (see \cite{Ozols}, for example).
Near a cleave point, the cut locus is a smooth hypersurface. The rest of the cut locus is
more complicated, but we know that $\Hnuno ( \tmop{Cut} \setminus
\tmop{Cleave} )=0$ and, indeed, that $\tmop{Cut} \setminus \tmop{Cleave}$ has
Hausdorff dimension at most $n-2$, for a smooth Riemannian manifold.

An isometric immersion from $M_{1} \setminus A$ into a complete manifold, for any set $A$ such that
$\Hnuno (A)=0$, can be extended to an isometric immersion from $M_{1}$. Thus,
it only remains to show that, for a cleave point $\cleaveq$, we have $e_{2}(x_{1} )=  e_{2} (x_{2} )$.

The way to do this is to find for each cleave point $q$ as above, a curve $Y$ whose image is contained in $\tmop{TCut}_{p_{1}}$ (in the metric space $\tmop{AC} (T_p M$) of absolutely continuous curves) such that $Y(0)=x_{1}$, $Y(1)=x_{2}$, and $ e_{1} \circ Y:[0,1] \rightarrow M_{1}$ is {\emph{fully tree-formed}}.

James Hebda proves in lemma 4.1 of \cite{Hebda} that this implies that $e_2(x_1)=e_2(x_2)$.
We extend that lemma in our lemma \ref{lem: si e_1 o Y es tree formed, sus extremos tienen la misma I_Y(.)}, so that it is simpler to use, and more general.
This is an important concept for this paper:

\begin{dfn}\label{def: linking curve}
  A \strong{linking curve} is an absolutely continuous curve $Y:[0,l]\rightarrow T_pM$ such that $e_1\circ Y$ is a fully tree formed curve.
\end{dfn}

\begin{dfn}\label{def: cut locus linking curve}
 A \strong{cut locus linking curve} is a linking curve $Y$ whose image is contained in the tangent cut locus, so that $e_1\circ Y$ is a fully tree formed curve with image contained in the cut locus.
\end{dfn}

J. Hebda's way to find the cut locus linking curves works only in dimension $2$.
Let $S_{p_{1}} M_{1}$ be the set of unit vectors in $\Tone$ parametrized with a coordinate $\theta$, and define $\rho :S_{p_{1}} M_{1} \rightarrow \mathbb{R}$ as the first cut point along the ray $t \rightarrow tv$ for $t>0$, and $\rho ( \theta)= \infty$ if there is no cut point on the ray.
The tangent cut locus is parametrized by $\theta\rightarrow( \rho ( \theta ), \theta)$, defined on the subset of $S_{p_{1}} M_{1}$ where $\rho$ is finite.
Given a cleave point $\cleaveq$, with $x_{i} =( \rho ( \theta_{i} ), \theta_{i}$), then $\rho$ is finite in at least one of the two arcs in $S_{p_{1}} M_{1}$ that join $\theta_{1}$ and $\theta_{2}$, which we write $[ \theta_{1} , \theta_{2} ]$.
Then the curve $Y(\theta )=( \rho ( \theta ), \theta)$ defined in $[ \theta_{1} , \theta_{2} ]$, satisfies the previous hypothesis.

It is important that $Y$ be absolutely continuous, which follows once it is proved that $\rho$ is.
This was shown independently in {\cite{Hebda94}} and {\cite{Itoh96}}, and later generalized to arbitrary dimension in
{\cite{Itoh Tanaka 00}}.

\subsection{Difficulties to extend the proof to dimension higher than $2$}

In dimension higher than $2$, there is no natural choice for a cut locus linking curve joining the speed vectors of the two minimizing geodesics that reach a cleave point.
Indeed, we can show that for some manifolds it is impossible to do so:

\begin{theorem}\label{thm: any manifold has a metric without cutLC}
Let $M$ be a smooth manifold of dimension $3$, and $p$ a point in $M$.
There is an open subset of the set of smooth Riemannian metrics such that for any cleave point $q=\exp_p(x_1)=\exp_p(x_2)$ from $p$, there is not CutLC whose extrema are $x_1$ and $x_2$.
\end{theorem}
\begin{proof}

Using the main theorem in {\cite{MR0221434}}, there is a metric $g_1$ on $M$ whose tangent cut locus from $p$ does not contain conjugate points (in other words, all segments with $P$ as one endpoint are non-conjugate).
Any metric sufficiently close to $g_1$ will also have disjoint cut and conjugate locus.

Let $q=\exp_p(x_1)=\exp_p(x_2)$ be a cleave point and $Y:[0,L]\rightarrow T_{p_1}M$ be a CutLC joining $x_{1}$ and $x_{2}$, where $T:[0,1]\rightarrow\Gamma$ is the identification map of $\exp_p\circ Y$.

We can change the parameter $t:[0,l]\rightarrow [0,L]$ so that $(u\circ t)(s)$ has unit speed (the identification is reparametrized accordingly $(T\circ t)(s)$).
We simply assume that the speed vector of $u$ has norm one wherever it is defined and keep the letter $t$ for the parameter.

Let $t_\ast=L/2$
The following possibilities may occur:
\begin{enumerate}
 \item There is some $t_0\neq t_\ast$ such that $T(t_0)=T(t_\ast)$.
 \item There is some $\varepsilon>0$ such that any point in $[t_\ast-\varepsilon, t_\ast+\varepsilon]$ is not identified to other points by $T$.
 \item There is a sequence $t_n\rightarrow t_\ast$ and a sequence $r_n\rightarrow t_\ast$ such that $r_n\neq t_n$ and $T(r_n)=T(t_n)$
 \item There is some $\varepsilon>0$, a sequence $t_n\rightarrow t_\ast$ and a sequence $r_n$ such that $|r_n-t_\ast|>\varepsilon$ and $T(r_n)=T(t_n)$.
\end{enumerate}

The second option is in contradiction with the hypothesis.
The reason is that for any continuous $1$-form $\varphi$ along $u|_{[t_\ast-\varepsilon, t_\ast+\varepsilon]}$, we must have
\[
  \int_{0}^{1} \varphi (s)(u' (s))ds=\int_{t_\ast-\varepsilon}^{t_\ast+\varepsilon} \varphi(s)(u' (s))ds= 0
\]
but since $T$ does not identify points in $[t_\ast-\varepsilon, t_\ast+\varepsilon]$ to other points, we can choose the continuous $1$-form $\varphi|_{[t_\ast-\varepsilon/2, t_\ast+\varepsilon/2]}$ freely, and this implies that $u'$ is null on that interval, which is in contradiction with having unit speed.

The fourth option implies the first, since a subsequence of the $r_n$ will converge to some $r_0$ which is not in $(t_\ast-\varepsilon, t_\ast+\varepsilon)$ and so is different from $t_\ast$.

If the third option holds, since $T(r_n)=T(t_n)$ implies $u(r_n)=u(t_n)$, or $\e(Y(t_n))=\e(Y(r_n))$, any neighborhood of $Y(t_\ast)$ contains a pair of different points with the same image, which implies that $\e$ is not a local diffeomorphism at $Y(t_\ast)$, in contradiction with the fact that the image of $Y$ is contained in the tangent cut locus, which does not contain conjugate points.

Only the first option remains.
In this case, it follows from definition \ref{def: tree formed curves} that the curve $Y|_{[t_0,t_\ast]}$ is tree formed, for the identification $T|_{[t_0,t_\ast]}$ (if $t_\ast<t_0$, we restrict to $[t_\ast,t_0] $).
The length of $[t_0,t_\ast]$ is smaller than $L/2$ and $Y|_{[t_0,t_\ast]}$ is also a CutLC.
We can iterate the argument to get a sequence of nested closed intervals whose length decreases to $0$.
The point in the intersection of that sequence is a conjugate point, by a similar argument as in the third option above, and this is again a contradiction.

\end{proof}

\section{Affine development and tree formed curves}\label{section: Affine development and tree formed curves}

In this section we extend the main results of sections 3 and 4 in \cite{Hebda}.

For this whole section, let $(M_1, p_1)$ and $(M_2, p_2)$ be two manifolds with $L$-related curvature.

\begin{dfn}
  
  The {\emph{local linear isometry induced by $x\in T_{p_1}M_1$}} is defined by
$$
I_x = P_{\gamma_2}\circ L \circ P_{-\gamma_1}
$$
  where $\gamma_1$ is the geodesic on $M_1$ with $\gamma_1(0)=x$, $\gamma_2$ is the geodesic on $M_2$ with $\gamma_2(0)=L(x)$ and $P_{\alpha}$ is the parallel transport along the curve $\alpha$.
\end{dfn}

\begin{remark}
 Since parallel transport along $\gamma\in AC_p(M)$ depends continuously on $\gamma$ (see \cite[6.1,6.3]{Hebda}), the map $x\rightarrow I_x$ is continuous.
\end{remark}

\begin{lem}\label{lem: I_X cuando X es punto regular}
  Let $x \in T_{p_1}M_1$ be a regular point of $e_1$, and $O$ any neighborhood of $x$ in $T_{p_1}M_1$ such that $e_1 |_{O}$ is injective.
  Let $f_x$ be the local isometry $e_2 \circ L \circ (e_1 |_{O_1})^{- 1}$ from $e_1 (O)$ in $M_1$ to $e_2 (O)$ in $M_2$.
  Then 
  $$I_x = d_{e_1(x)} f_x$$%
\end{lem}

\begin{proof}
  See lemma 1.35 in \cite{Cheeger Ebin}
\end{proof}

We define $\tmop{Dev}_i : \tmop{AC} (M_i) \rightarrow \tmop{AC} (T_{p_i} M_i)$ as the affine development of absolutely continuous curves in $M_i$ based at $p_i$, for $i = 1, 2$.

\begin{lem}\label{Dev_2^-1 o L o Dev_1(u) = v}
  Let $Y : [0, l] \rightarrow V_1$ be an absolutely continuous curve such that $Y (0) = 0$.
  Then:
\begin{enumerate}
 \item $I_{Y(l)} = P_{v} \circ L \circ P_{-u}$.
 \item At any point $t$ where $u'(t)$ and $v'(t)$ are defined, $I_{Y(t)}(u'(t))=v'(t)$.
 \item $v = (\tmop{Dev}_2)^{- 1} \circ L \circ \tmop{Dev}_1 (u) $.
\end{enumerate}
\end{lem}

\begin{proof}
  We first assume that the image of the curve $Y$ is contained in the interior of $V_1$.
  Notice that if $Y$ is a radial line, the first statement is just the definition of $I_{Y(l)}$.
  Define:
  \[
    J = \left\{ t : L \circ P_{-u|_{[0,t]}} = P_{-v|_{[0,t]}} \circ I_{Y (t)} \right\}
  \]
  We will prove that $J = [0, l]$ by proving it is open and close.
  
  If $[0, t) \subset J$, we take a sequence $t_j \nearrow t$ to find by continuity of parallel transport and $L \circ P_{-u_{[0,t_j]}} = P_{-v_{[0,t_j]}} \circ I_{Y (t_j)} $ that $L \circ P_{-u_{[0,t]}} = P_{-v_{[0,t]}} \circ I_{Y (t)} $, so closedness follows.
  
  Assume now $[0, t] \subset J$. $Y (t)$ is in the interior of $V_1$ by hypothesis, so there is $\varepsilon>0$ and a neighborhood $O$ of $Y|_{[t-\varepsilon, t+\varepsilon]}$, and an isometry $\varphi:e_1(O)\rightarrow e_2(O)$ with $\varphi\circ e_1|_O = e_2|_O$.
  Then for any $t<t_1<t+\varepsilon$
  $$P_{-u|_{[0,t_1]}} =P_{-u|_{[0,t]}}\circ P_{-u|_{[t,t_1]}}$$
  and similarly for $v$.
  By hypothesis, 
  $$L\circ P_{-u|_{[0,t]}} = P_{-v|_{[0,t]}}\circ I_{Y (t)}.$$
  We have $\varphi\circ e_1|_O = e_2|_O$ so, as parallel transport commutes with isometries, we have 
  $$
  \begin{array}{rcl}
  P_{-v|_{[t,t_1]}}\circ I_{Y(t_1)} &=& P_{-v|_{[t,t_1]}}\circ d_{u(t_1)}\varphi \\
  &=& d_{u(t)}\varphi \circ P_{-u|_{[t,t_1]}} \\
  &=& I_{Y(t)}\circ P_{-u|_{[t,t_1]}}
  \end{array}
  $$
  It follows that $[t,t+\varepsilon)\subset J$, so $J$ is open and the first item follows when the image of $Y$ is contained in the interior of $V_1$.

  We next prove that $I_{Y(t)}(u'(t))=v'(t)$ for any $t\in [0,l]$ such that $Y'(t)$ is defined.
  This is clear when $Y(t)$ is in the interior of $V_1$, because, for an isometry $f$ defined in a neighborhood $O$ of $Y(t)$ where $e_1|_{O}$ is injective and $Y([t-\varepsilon, t+\varepsilon])\subset O$, we have $v|_{[t-\varepsilon, t+\varepsilon]}=f\circ u|_{[t-\varepsilon, t+\varepsilon]}$ and $I_{Y(t)}=d_{u(t)}f$.

  We now deal with curves whose image intersects the boundary of $V_1$.
  Such a curve $Y$ can be approximated in $AC_0(T_p M)$ by curves $Y_k(t)=(1-\frac{1}{k})Y(t)$ that stay in $int(V_1)$, so that $|Y_k-Y|_{AC_0(T_p M)}\rightarrow 0$.
  Taking limits as $k$ goes to infinite, the first item follows by continuity of parallel transport, the second because $x\rightarrow I_x$ is continuous and by a standard use of the chain rule.
  
  The third claim is equivalent to $\tmop{Dev}_2 (v) = L \circ \tmop{Dev}_1 (u)$, and this follows by integration if we prove 
  \begin{equation}
    L (P_{-u|_{[0,t]}} (u' (t))) = P_{-v|_{[0,t]}} (v' (t))  \hspace{1em} \text{ for almost every }t \in [0, 1] \label{eq: parallel transport of speed coincides}
  \end{equation}
  But this clearly follows from the two earlier items.

\end{proof}

\begin{remark}
  We will only need the above lemma, but it is worth mentioning that the above
  also holds for a more general path $Y : [0, 1] \rightarrow T_{p_1} M_1$.
  There are (at least) two ways to do it:
  \begin{enumerate}
    \item As the set of singular points of $e_1$ is a Lipschitz multi-graph
    (see Theorem A of \cite{Itoh Tanaka 00}), we can approximate $Y$ by
    paths that are transverse to the set of conjugate points. The proof that
    $J$ is open at an intersection point $t_0$ consists of gluing two intervals
    $(t_0 - \varepsilon, t_0)$ and $(t_0, t_0 + \varepsilon)$ where $e_1$ is
    not singular, and continuity of $I_X$ makes the gluing possible.
    
    \item The above approach is straightforward but poses some technical
    difficulties. An alternative approach is to approximate the metric by a
    generic one and $Y$ by a generic path in $T_{p_1} M_1$.
    The manifold $M_1$ with the new metric will no longer have curvature $L$-related to that of $M_2$, but the local maps $I_X$ can still be defined as a continuous family of linear isomorphisms.
    The path $Y$ will cross the set of conjugate points transversally, and only at $A_2$
    singularities, which will simplify the proof.
  \end{enumerate}
\end{remark}

\begin{lem}\label{lem: si e_1 o Y es tree formed, sus extremos tienen la misma I_Y(.)}
  Let $Y:[0,l]\rightarrow T_{p_1}M_1$ be a linking curve whose image is contained in $V_1$. Then:
  \begin{itemize}
   \item $e_2 (Y (0)) = e_2 (Y (l))$.
   \item $I_{Y(0)} = I_{Y(l)}$
  \end{itemize}
\end{lem}

\begin{proof}
  Let $r:[0,1]\rightarrow V_1$ be the radial path from $0$ to $Y(0)$.
  Define $u = e_1 (r \ast \alpha)$ and $v = e_2 (r \ast \alpha)$, which are absolutely continuous curves defined on the interval $[0,l+1]$.
  Then $v = \tmop{Dev}_2^{- 1} \circ L \circ \tmop{Dev}_1 (u)$ by the previous lemma.
  
  By its definition, $u$ is tree-formed for an identification map $T$ with $T
  \left( 1 \right) = T (l+1)$, so it follows that $v$ also is, by theorem \ref{Hebda3.3}.
  It follows that $e_2 (Y (0)) =v(1)=v(l+1)= e_2 (Y (l))$.
  
  For the second claim, observe that:
\[
I_{Y(l)} = P_{e_2\circ(r\ast Y)} \circ L \circ P_{-e_1\circ (r\ast Y)}=
P_{e_2\circ r} \circ P_{e_2\circ Y} \circ L \circ P_{-e_1\circ Y} \circ P_{-e_1\circ r}
\]
We can simplify this expression, since both $e_1\circ Y$ and $e_2\circ Y$ are fully tree formed:
\[
I_{Y(l)} =
P_{e_2\circ r} \circ L \circ P_{-e_1\circ r}
\]
and that is the definition of $I_{Y(0)}$.
\end{proof}

\section{Synthesis}\label{section: synthesis}

For any point $x\in int(V_1)$, the Cartan lemma provides an isometry from a neighborhood of $e_{1} (x)$ to one of $e_{2} (x)$.
Our plan is to collect those local mappings to build a covering space.%

\begin{dfn}
  A Riemannian covering is a local isometry that is also a covering map (see \cite{O'Neill68} for a motivation).
\end{dfn}

\begin{dfn}\label{dfn: synthesis}
Let $A$ be a topological manifold, $X_{1}$, $X_{2}$ are Riemannian manifolds, and $e_{1} :A \rightarrow X_{1}$, $e_{2} :A \rightarrow X_{2}$ are continuous surjective maps.

  A {\strong{synthesis}} of $X_{1}$ and $X_{2}$ through $e_1$ and $e_2$ is a Riemannian manifold $X$, together with a continuous map $e :A \rightarrow X$, 
  and Riemannian coverings $\pi_{i} :X \rightarrow X_{i}$ for $i=1,2$ such
  that $\pi_{i} \circ e = e_{i}$.

  $$
  \xymatrix{   & A \ar@//[ddl]_{e_1} \ar@//[ddr]^{e_2} \ar@//[d]^{e} \\
  & X \ar[dl]^{\pi_1} \ar[dr]_{\pi_2} &  \\ X_1 &  & X_2}
  $$
If $\pi_i$ are only local isometries, then $X$ is called a \strong{weak synthesis}.
\end{dfn}

We will use this extension of the Ambrose conjecture in terms of synthesis manifolds (see section 3 in \cite{O'Neill68}).

\begin{conjecture}[\textsc{Ambrose Conjecture following O'Neill}]\label{ambrose conjecture}
  Let ($M_{1} ,p_{1}$) and ($M_{2} ,p_{2}$) be two Riemannian manifolds with $L$-related curvature.
  Define $e_1=\exp_{p_1}$ and $e_2=L\circ\exp_{p_2}$.
  
  Then there is a synthesis $M$ of $M_{1}$ and $M_{2}$ through $e_1$ and $e_2$, and a point $p\in M$ such that $\pi_i(p)=e_i(0)$ for $i=1,2$.

  In particular, if $M_{1}$ and $M_{2}$ are simply-connected, then $\pi_{2} \circ \pi_{1}^{-1} :(M_{1} ,p_{1} )   \rightarrow (M_{2} ,p_{2} )$ is the unique isometry whose tangent at $p_{1}$ is $L$.
\end{conjecture}

If $ e_{1}$ has no singularities, we can pull the metric from $M_{1}$ onto $T_{p_{1}} M_{1}$ and the desired Riemannian 	coverings are $\pi_1=e_{1}$ and $\pi_2=e_{2}$.
In the presence of singularities, \emph{the idea is to build the synthesis as a quotient of $A=V_{1}$} that identifies pairs of points with the same image by both $ e_{1}$ and $ e_{2}$.

\subsection{Unequivocal points and linked points}\label{section: unequivocal y linked}

\begin{dfn}\label{definition:unequivocal}
  We say that an open set $O \subset V_1$ is {\strong{unequivocal}} if and only if $ e_{1} (O)$ is an open set, and there is an isometry $\varphi_O: e_1(O)\rightarrow e_2(O)$ such that $\varphi_O\circ e_1|_O=e_2|_O$.
\end{dfn}

\begin{dfn}\label{definition:unequivocal point}
  We say $x \in V_1$ is {\strong{unequivocal}} if there is a sequence of unequivocal sets $W_n$ such that $e_1(W_n)$ is a neighborhood base of $e_1(x)$.
\end{dfn}

\begin{remark}
 The above definition allows for points $x\in V_1$ that are not isolated in $e_1^{-1}(e_1(x))$.
 This is important if we want a definition of sutured manifold that may hold for all Riemannian manifolds.
\end{remark}

We plan to identify points in $T_pM$ that are joined by a linking curve.
However, in order to build a quotient space, we need some kind of \emph{openness} as in Lemma~\ref{linked to unequivocal is open}.
In order to define a \emph{relaxed version} of the above relation for which Lemma~\ref{linked to unequivocal is open} holds, we need to allow curves with some sort of ``controlled'' discontinuities.

\begin{dfn}\label{dfn: quasi-continuous linking curve}
   A \emph{quasi-continuous linking curve} is a bounded curve $Y:[0,l]\rightarrow V_1$ such that:
\begin{enumerate}
 \item The composition $e_1\circ Y$ is an absolutely continuous tree formed curved.
 \item For every point $t_0$, there is an $\varepsilon>0$ such that either $Y|_{[t_0-\varepsilon, t_0+\varepsilon]}$ is absolutely continuous, or its image is contained in an unequivocal set $W$.
\end{enumerate}
\end{dfn}

\begin{dfn}\label{def: linked points}
Two points $x, y \in V_1$ are {\strong{strongly linked}} (by the curve $Y$) iff there is a linking curve $Y:[0,l]\rightarrow V_1$ such that $Y(0)=x$ and $Y(l)=y$.

Two points $x, y \in V_1$ are {\strong{linked}} ($x\leftrightsquigarrow y$) if and only if there is a quasi-continuous linking curve $Y:[0,l]\rightarrow V_1$ such that $x$ is the limit of $Y(t_j)$ for some sequence $t_j\searrow 0$ and $y$ is the limit of $Y(t_j)$ for some sequence $t_j\nearrow l$.

\end{dfn}

\subsection{Main properties of unequivocal sets and linked points}

In this section we extend some results from section \ref{section: Affine development and tree formed curves}.

\begin{lem}\label{lem: I_X cuando X es unequivocal}
  Let $W$ be any unequivocal neighborhood of $x \in T_{p_1}M_1$.
  Let $\varphi:e_1(W)\rightarrow e_2(W)$ be the local isometry such that $\varphi\circ e_1=e_2$.
  Then 
  $$I_x = d_{e_1 (x)} \varphi$$%
  In particular, it depends only on $e_1(x)$.
\end{lem}

\begin{proof}
  If $x$ is a regular point of $e_1$, we know from lemma \ref{lem: I_X cuando X es punto regular} that $I_x=d_{e_1 (x)} f_x$.
  Both $f_x$ and $\varphi$ are isometries that agree on the open set $e_1(U)$, for an open set $U\subset W$ such that $e_1$ is injective when restricted to $U$.
  Thus $f_x$ and $\varphi$ agree on $U$ and the result follows.

  For a conjugate point $x$, we take limits of a sequence of regular points, since $I_z = d_{e_1 (z)} \varphi$ for any regular point $z\in W$, and $z\rightarrow I_z$ is continuous.
\end{proof}

\begin{lem}\label{Dev_2^-1 o L o Dev_1(u) = v, para Y con teletransporte}
  Let $Y : [0, l] \rightarrow V_1$ be a bounded curve such that:
\begin{itemize}
 \item   $Y (0) = 0$.
 \item  $u = e_1 \circ Y$ and $v = e_2 \circ Y$ are absolutely continuous.
 \item For every point $t_0$, there is an $\varepsilon>0$ such that either $Y|_{[t_0-\varepsilon, t_0+\varepsilon]}$ is absolutely continuous, or its image is contained in an unequivocal set $W$.
\end{itemize}
  Then:
\begin{enumerate}
 \item $I_{Y(l)} = P_{v} \circ L \circ P_{-u}$.
 \item At any point $t$ where $u'(t)$ and $v'(t)$ are defined, $I_{Y(t)}(u'(t))=v'(t)$.
 \item $v = (\tmop{Dev}_2)^{- 1} \circ L \circ \tmop{Dev}_1 (u) $.
\end{enumerate}
\end{lem}

\begin{proof}
 Define $J$ as in lemma \ref{Dev_2^-1 o L o Dev_1(u) = v}: 
  \[
    J = \left\{ t : L \circ P_{-u|_{[0,t]}} = P_{-v|_{[0,t]}} \circ I_{Y (t)} \right\}
  \]
 
 We know that $I_{Y (t)}$ is continuous at every $t$ where $Y$ is continuous.
 By the last hypothesis and the previous lemma, $I_{Y (t)}$ is also continuous at points where $Y$ is discontinuous.
 It follows that $J$ is closed and it remains to prove that it is open.
 Let $t_0\in J$.
 
 If $Y|_{[t_0-\varepsilon, t_0+\varepsilon]}$ is absolutely continuous and its image is contained in $int(V_1)$, we prove that $[t_0,t_0+\varepsilon)\subset J$ as we did in Lemma~\ref{Dev_2^-1 o L o Dev_1(u) = v}.
 If $Y|_{[t_0-\varepsilon, t_0+\varepsilon]}$ is contained in an unequivocal set $W$, there is an isometry $\varphi:e_1(W)\rightarrow e_2(W)$ such that $\varphi\circ u|_{[t_0-\varepsilon, t_0+\varepsilon]} = v|_{[t_0-\varepsilon, t_0+\varepsilon]}$ so, as parallel transport commutes with isometries, and using Lemma~\ref{lem: I_X cuando X es unequivocal}, we have, for $t_0<t_1<t_0+\varepsilon$ 
  $$
  \begin{array}{rcl}
  P_{-v|_{[t_0,t_1]}}\circ I_{Y(t_1)} &=& P_{-v|_{[t_0,t_1]}}\circ d_{u(t_1)}\varphi \\
  &=& d_{u(t_0)}\varphi \circ P_{-u|_{[t_0,t_1]}} \\
  &=& I_{Y(t_0)}\circ P_{-u|_{[t_0,t_1]}}
  \end{array}
  $$
 and we deduce that $[t_0,t_0+\varepsilon]\subset J $ as in Lemma~\ref{Dev_2^-1 o L o Dev_1(u) = v}.
 
 Finally, if $Y|_{[t_0-\varepsilon, t_0+\varepsilon]}$ is absolutely continuous but its image is not contained in $int(V_1)$,  we define for every $k$ a modified curve:

$$
Y_{t_0,\varepsilon,k}(t) = \left\{ \begin{array}{ll}
  Y(t)& t\leq t_0 \\
  \left( 1 -\frac{1}{\varepsilon k}(t-t_0)\right)Y(t) & t_0<t\leq t_0+\varepsilon\\
  \left( 1 -\frac{1}{k}\right)Y(t) & t_0+\varepsilon<t
\end{array} \right.
$$
 Since $Y_{t_0,\varepsilon,k}|_{[t_0-\varepsilon, t_0+\varepsilon]}$ is absolutely continuous and its image is contained in $int(V_1)$, we learn that for any $t<t_0+\varepsilon$
$$
L \circ P_{-u_k|_{[0,t]}} = P_{-v_k|_{[0,t]}} \circ I_{Y_{t_0,\varepsilon,k} (t)}
$$
 and since $Y_k$ converges to $Y$ in $AC_0(T_p M)$, we have proven that $[t_0, t_0+\varepsilon)\subset J$.
 
 We now turn to the proof that $I_{Y(t)}(u'(t))=v'(t)$ for almost every $t\in [0,1]$.
 We have already shown this if $Y|_{[t-\varepsilon, t+\varepsilon]}$ is absolutely continuous and $Y(t)$ belongs to $int(V_1)$.
 If $Y|_{[t-\varepsilon, t+\varepsilon]}$ is absolutely continuous but $Y(t)$ does not belong to $int(V_1)$, we construct the same curves $Y_{t_0,\varepsilon,k}$: we know that for every $t\in [0,1]$ for which $u'(t)$ and $v'(t)$ are defined, we have $I_{Y_{t_0,\varepsilon,k}(t)}(u'(t))=v'(t) $.
 Since $I_{Y_{t_0,\varepsilon,k}(t)}$ converges to $I_{Y(t)}$ as $k$ goes to infinity, it follows that $I_{Y(t)}(u'(t))=v'(t)$.
 
 Finally, if $Y|_{[t-\varepsilon, t+\varepsilon]}$ is contained in an unequivocal set $W$, let $\varphi:e_1(W)\rightarrow e_2(W)$ be the isometry in the definition of unequivocal set.
 By lemma \ref{lem: I_X cuando X es unequivocal}
$$
I_{Y(t)}(u'(t_0))=(d_{u(t_0)}\varphi)(u'(t_0))=(\varphi\circ u)'(t_0) = v'(t_0)
$$

 The third item follows from the first and the second as in Lemma~\ref{Dev_2^-1 o L o Dev_1(u) = v}.

\end{proof}

\begin{lem}\label{lem: si x e y linked, I_x=I_y}
 Let $x,y\in V_1$ be \emph{linked} points.
 Then
 \begin{enumerate}
  \item $e_1(x)=e_1(y)$
  \item $e_2(x)=e_2(y)$
  \item $I_{x}=I_{y}$
 \end{enumerate}
\end{lem}
\begin{proof}
 Let $Y$ be a quasi-continuous linking curve that links $x$ and $y$.

 The first part is obvious from the definition because $e_1(x)$ and $e_1(y)$ are the extrema of the fully tree formed curve $e_1\circ Y$.
 
 The second and third parts follow as in \ref{lem: si e_1 o Y es tree formed, sus extremos tienen la misma I_Y(.)}, because the curve $r\ast Y$ satisfies the hypothesis of lemma \ref{Dev_2^-1 o L o Dev_1(u) = v, para Y con teletransporte}.
\end{proof}

\begin{lem}\label{linked to unequivocal is open}
 Let $x\in V_1$ be linked to some $z\in W$ for an unequivocal set $W$.
 Then there is a neighborhood $U\subset V_1$ of $V_1$ that contains $x$ and such that every $y\in U$ is linked to some $w\in W$.
\end{lem}
\begin{proof}
  We define $U$ as the connected component of $e_1^{-1}(e_1(W))\cap V_1$ that contains $x$.
  For $y\in U$, we want to prove that $y$ is linked to some $w\in W$.
  
  Let $Z:[0,l]$ be a quasi-continuous linking curve that joins $x$ and $z$.
  We want to find curves $A:[0,a]\rightarrow U$ and $B:[0,a]\rightarrow W$ such that $e_1\circ(A\ast Z\ast B)$ is fully tree formed, $A(0)=y$.
  This holds if we choose an arbitrary absolutely continuous path $A$ with $A(0)=x$ and $A(a)=y$, and $B(t)$ so that $e_1(B(t))=e_1(A(a-t))$.
  We may not be able to choose an absolutely continuous path $B$, but since its image is contained in $W$, $Y$ is a quasi-continuous linking curve.

\end{proof}

\begin{remark}
 Such a choice of $B$ is not very elegant, and requires using the axiom of choice.
 The interested reader can find a more constructive alternative in the proof of Proposition~\ref{d is distance-decreasing}.
\end{remark}

\subsection{Construction of the synthesis}\label{subsection: synthesis}

\begin{theorem}

  \label{thm: weak synthesis}Let $M_{1}$, $M_{2}$ be Riemannian manifolds with $L$-related curvature, such that for every $x \in V_1$ is linked to some unequivocal point $y\in V_1$.
  
  Then there is a \emph{weak synthesis} of $M_{1}$ and $M_{2}$ through $e_1$ and $e_2$.%

\end{theorem}

\begin{proof}
  Define a set $M$ as a quotient by the equivalence relation:
  \begin{equation*}
    M = \left(A/ \leftrightsquigarrow \right)
  \end{equation*}
  Let $e:A \rightarrow M$ be the projection map. We define maps $\pi_{ i} :M
  \rightarrow M_{i}$ by $\pi_{i} ([x])=e_{i} (x)$. Both maps are well defined by lemma \ref{lem: si x e y linked, I_x=I_y}.
  
  We give $M$ a topology where the basic open sets are $[W]=\{[x],x \in W\}$, for unequivocal open set $W$.
  \begin{itemize}
    \item By hypothesis, every point belongs to some open set, so this is a good basis for a topology.
    
    \item $e$ is continuous at every point $x\in A$:
    Let $[W]$ be a basis open neighborhood of $[x]$, for $W$ unequivocal.
    There is $z\in W$ such that $x\leftrightsquigarrow z$, by a quasi-continuous linking curve $\rho$.
    By lemma \ref{linked to unequivocal is open}, there is an open neighborhood $U$ such that any point in $U$ is linked to some point in $W$.
    Thus $U$ is contained in $e^{-1}([W])$.

    \item $\pi_{1} |_{[W]}$ is injective for any basis open set $[W]$: 
    Let $[x_{1} ],[x_{2} ] \in [W]$ be such that $\pi_1 ([x_{1} ])= \pi_1 ([x_{2} ])$.
    This means $e_1(x_1)=e_1(x_2)$.
    We can assume $x_{1} ,x_{2} \in W$, which implies $[x_1]=[x_2]$ (using a curve $Y$ that only takes the values $x_1$ and $x_2$).
    
    \item $\pi_{2} |_{[W]}$ is injective for any basis open set $[W]$: 
    If $\pi_2 ([x_{1} ])= \pi_2 ([x_{2} ])$ for $x_1,x_2\in W$, it follows that $e_2(x_1)=e_2(x_2)$, which implies $\varphi_W(e_1(x_1))=\varphi_W(e_1(x_2))$, for the isometry $\varphi_W$ in the definition of unequivocal set, which implies $e_1(x_1)=e_1(x_2)$ and $[x_1]=[x_2]$.

    \item $\pi_{1} |_{[W]}$ is continuous, for a basis set $[W]$: let $U$ be an open
    subset of $\pi_{1} ([W])=e_1(W)$. Then $( \pi_{1} |_{[W]} )^{-1} (U)=[W] \cap
    \pi_{1}^{-1} (U)=[W \cap e_{1}^{-1} (U)]$ is an open set, because $e_{1} (W \cap e_{1}^{-1} (U))= e_1(W)\cap U$ is an open set and $W \cap e_{1}^{-1} (U) \subset W$, so $W \cap e_{1}^{-1} (U)$ is unequivocal.
    
    \item $\pi_{2} |_{[W]}$ is continuous, for a basis set $[W]$: let $U$ be an open subset of $\pi_{2} ([W])=e_2(W)$, and let $\varphi_W:e_1(W)\rightarrow e_1(W)$ be the isometry associated with $W$.
    Then $( \pi_{2} |_{[W]} )^{-1} (U)=[W] \cap \pi_{2}^{-1} (U)=[W \cap e_{2}^{-1} (U)]$.
    This is an open set, because $e_{1} (W \cap e_{2}^{-1} (U))=\varphi_W^{-1}(e_{2} (W \cap e_{2}^{-1} (U))=\varphi_W^{-1}(e_{2}(W)\cap U) =\varphi_W^{-1}(\varphi_W(e_1(W))\cap U)$ is an open set and $W \cap e_{2}^{-1} (U) \subset W$, so $W \cap e_{2}^{-1} (U)$ is unequivocal.
    
    \item For a basis open set $[W]$, $\pi_1 ([W])=e_1(W)$ is open by hypothesis, and $\pi_2([W])=e_2(W)=\varphi(e_1(W))$ is also open.
    Hence, $\pi_{i}$ is open for $i=1,2$.
    Thus, $\pi_{i} |_{[W]}$ is an homeomorphism onto its image.
    
    \item Since $\pi_{1}$ and $\pi_{2}$ are local homeomorphisms, we can use $\pi_{1}$, for instance, to give $M$ the structure of a smooth Riemannian manifold, which trivially makes $\pi_{1}$ a local isometry.
    For an unequivocal set $W$, with $e_{2} |_{W} = \varphi \circ e_{1} |_{W}$, then $\pi_{2} \circ (    \pi_{1} |_{[W]} )^{-1} = \varphi$ is an isometry from $\pi_{1} ([W])=e_{1}(W)$ onto $\pi_{2} ([W])=e_{2} (W)$, so $\pi_{2}$ is also a local isometry.

  \end{itemize}
\end{proof}

\subsection{Compactness}\label{subsection: from local homeo to covering}

In order to prove Theorem \ref{main theorem ambrose}, we still have to prove that $\pi_{1}$ and $\pi_{2}$ given by theorem \ref{thm: weak synthesis}  are covering maps.
This requires some sort of ``compactness'' result, and using the extra hypothesis in the definition of a sutured manifold.
We start with a general lemma:

\begin{lem}
  \label{norm dominated by exp}Let $\e : \T \rightarrow M$ be the exponential
  map from a point $p$ in a Riemannian manifold $M$. Then for any absolutely continuous
  path $x:[0,t_0] \rightarrow \T$, the total variation of $t \rightarrow |x(t)|$
  is not greater than the length of $t \rightarrow \e (x(t))$. In particular:
  \[ |x(t_0)|-|x(0)|< \tmop{length} ( \e \circ x) \]
\end{lem}

\begin{proof}
 For an absolutely continuous path $x$:

  \[ \tmop{length} ( \e \circ x ) = \int_0^{t_{0}} |( \e \circ x )' |= \int_0^{t_{0}} |d \e ( x' )| \]
  The speed vector $x' =ar+v$ is a linear combination of a multiple
  of the radial vector and a vector $v$ perpendicular to the radial direction.
  By the Gauss lemma, $|d \e ( x' )|  = \sqrt{a^{2} +|d \e (v)|^{2}} \geq |a|$.
  On the other hand, $v$ is tangent to the spheres of constant radius, so:
  \[ TV_0^{t_{0}} (|x|)= \int_0^{t_{0}} \left|\frac{d}{dt} |x|\right|= \int_0^{t_{0}} |a| \leq \tmop{length} ( \e \circ x) \]

\end{proof}

Let us now come back to our hypothesis.

\begin{dfn}\label{dfn: distance to p in the synthesis}
 Let $(M_1,p_1)$ is sutured, and $(M_2,p_2)$ a manifold with $L$-related curvature.
 Let $M$ be the weak synthesis obtained by application of Theorem \ref{thm: weak synthesis} and $p=e(0)\in M$.

 Then the \strong{synthesis-distance} to $p$ is the function $d:M \rightarrow \mathbb{R}$ given by
\[ d(q)= \inf_{x \in e^{-1} (q)} \{|x|\} \]
\end{dfn}

If we could prove that $e$ is the exponential map of the Riemannian manifold $M$ at the point $p=e ( 0 )$, it would follow that $d$ is the distance to $p$, and the following proposition would be trivial.

\begin{prop} \label{d is distance-decreasing}
 The synthesis-distance $d$ is a $1$-Lipschitz function on $M$: %
  \[ | d(q_{2} )-d(q_{1} ) | \leq d_{M} (q_{1} ,q_{2} ) \]
\end{prop}

\begin{proof}

\strong{Claim.}
  Given $q_1,q_2\in M$ and $\varepsilon>0$, there is a family of absolutely continuous paths $\beta_k:[0,l_k]\rightarrow V_1$, for $k$ integer, with the following properties:
\begin{itemize}
 \item the curves $\beta_k$ are parametrized so that $e_1\circ \beta_k$ has unit speed. This is equivalent to asking that $e\circ \beta_k$ has unit speed, since $\pi_1$ is a local isometry. In particular, $l_k=\tmop{length}(e_1\circ \beta_k)=\tmop{length}(e\circ \beta_k)$.
 \item $\beta_1(0)\in e^{-1}(q_1)$ and $|\beta_1(0)| < d(q_1) + \varepsilon/2$.
 \item $e(\beta_k(l_k))\rightarrow q_2$.
 \item for each $k$: $|\beta_{k+1}(0)|\leq |\beta_k(l_k)|$.
 \item $\sum_{k=1}^{\infty} l_k< d_M(q_1,q_2) + \varepsilon/2$.
\end{itemize}
The family of curves may be finite or infinite.
We will assume the latter, since the former is strictly simpler.

From this and Lemma~\ref{norm dominated by exp} it follows that 
\begin{equation}\label{bound for the norm of the lifted curve}
\begin{array}{rcl}
|\beta_N(l_N)| &=& |\beta_1(0)| + \sum_{k=1}^{N-1} (|\beta_{k+1}(0)| - |\beta_k(0)|) + (|\beta_N(l_N)| - |\beta_N(0)|)\\
      &<& |\beta_1(0)| + \sum_{k=1}^N (|\beta_k(l_k)| - |\beta_k(0)|) \\
      &<& |\beta_1(0)| + \sum_{k=1}^N l_k \\
      &<& d(q_1) + d_M(q_1,q_2) + \varepsilon
\end{array} 
\end{equation}
Thus the points $\beta_N(l_N)$ are bounded and a subsequence of them converge to some $x\in V_1$ which belongs to $e^{-1}(q_2)$ and satisfies the same bound.
This proves the result, and it remains to prove the claim.

By \cite[1.1]{Hebda}, $e_1(Sing\setminus \mathcal{A}_2)$ has null $\Hnuno$ measure.
We define $\mathcal{N}=e(Sing\setminus \mathcal{A}_2)$.
It follows that $\mathcal{N}$ has null $\Hnuno$ measure because $\mathcal{N}\subset \pi_1^{-1}(e(Sing\setminus \mathcal{A}_2))$, and the image of a $\Hnuno$-null set by a local isometry is also $\Hnuno$-null.%

Let $R=d(q_1) + d_M(q_1,q_2) + \varepsilon$, and let $B_0(R)\subset\T$ be the open ball of radius $R$.
The set $\mathcal{A}_2\cap B_0(R)$ is a smooth $n-1$-manifold (non-compact without boundary).
Any $x\in \mathcal{A}_2$ has a neighborhood $U$ such that $e_1(U\cap\mathcal{A}_2\cap B_0(R))$ and $e(U\cap\mathcal{A}_2\cap B_0(R))$ are (isometric) smooth $n-1$-manifolds.

We start the construction of $\beta$ choosing a starting point $\beta_1(0)\in e^{-1}(q_1)$ such that $|\beta_1(0)| < d(q_1) + \varepsilon/2$.
The point $\beta_1(0)$ may be singular (that is, $\beta_1(0)\in \partial V_1$), in which case we start $\beta_1$ with a short straight path that reaches a new $\beta_1(\varepsilon/4)\in int(V_1)\cap B_0(R)\setminus e_1^{-1}(e_1(Sing)\cap B_0(R))$.
We can choose $\beta_1$ to that its derivative makes a positive angle with the kernel of the exponential, and this is all we need to change the parameter so that $e_1\circ\beta_1$ has unit speed.
So we assume that the length of $e_1\circ\beta_1$ is $\varepsilon/4$.

By \cite[4.3]{Hebda82}, if $q_2\notin\mathcal{N}$, we can find a curve $c$ disjoint from $\mathcal{N}$ joining $e(\beta_0(\varepsilon/4))$ with $q_2$ whose length is not greater than $d_M(e(\beta_1(\varepsilon/4)), q_2) +\varepsilon/4<d_M(q_1,q_2)+\varepsilon/2$.
We remark that \cite[4.3]{Hebda82} requires that $M$ is complete, something that we have not proved yet.
However, the proof of \cite[4.3]{Hebda82} is valid also without this hypothesis with minor modifications:
\begin{itemize}
 \item let $v$ be a path in $M$ joining $e(\beta_1(\varepsilon/4))$ and $q_2$, of length smaller than $d_M(e(\beta_1(\varepsilon/4)),q_2)+\varepsilon/8 $
 \item find a finite partition $0=t_0<t_1<\dots<t_N $ of the domain of the curve so that two consecutive points $v(t_i), v(t_{i+1})$ lie in a strongly convex open set.
 \item choose points $c(t_0)=v(t_0)$, $c(t_N)=v(N)$ and $c(t_i)$ using \cite[4.2]{Hebda82} (for $K=\mathcal{N}$) so that the length of $c|_{[t_i,t_{i+1}]}$ is smaller than the length of $v|_{[t_i,t_{i+1}]}+\frac{\varepsilon}{8\,N}$.
\end{itemize}
The resulting curve $c$ does not intersect $\mathcal{N}$ and has length smaller than $length(v)+\varepsilon/8<d_M(e(\beta_1(\varepsilon/4)),q_2)+\varepsilon/4$.

If $q_2\in\mathcal{N}$, we can use a similar procedure: take a curve $v:[\varepsilon/8,0]\rightarrow M$ from a nearby point $v(\varepsilon/8)$ into $q_2$ of length smaller than $\varepsilon/8$ and split it by intervals of length $\varepsilon/2^{k+1}$ (we start from $k=3$ for convenience).
We can then replace $v$ by a broken geodesic $c$ that avoids $\mathcal{N}$ such that the length of the segment $c|_{[\varepsilon/2^k,\varepsilon/2^{k+1}]}$ is no more than $\varepsilon/2^k$.
In this way we find a continuous curve $c$ of length smaller than $\varepsilon/4$ that joins $q_2$ to a point not in $\mathcal{N}$.

We also want a curve that is transverse to $e(\mathcal{A}_2\cap B_0(R))$.
This is equivalent to being transverse to each of the countably many smooth manifolds $e(U\cap\mathcal{A}_2\cap B_0(R))$ that we mentioned before.
Since transversality to a smooth manifold is a residual property \cite[3.2.1]{Hirsch}, and a countable intersection of residual sets is residual, and in particular dense, we can find a curve $u:[0,l]\rightarrow M$ joining $\beta_1(\varepsilon/4)$ and $c(\varepsilon/8)$ that is close to $c$ in the $C([0,l],M)$ topology, so that, in particular, the length of $u$ is not greater than $d_M(q_1,q_2) +\varepsilon/2$, that is transverse to $e(\mathcal{A}_2\cap B_0(R))$ and does not intersect $\mathcal{N}$ except possibly at the final point.

Assume that the intersection points of $u$ and $e(\mathcal{A}_2\cap B_0(R))$ cluster at a point $u(t_\ast)\neq u(l)$.
Then there is a sequence of points $x_j\in \mathcal{A}_2\cap B_0(R)$ and times $t_j\rightarrow t_\ast$ such that $e(x_j)=u(t_j)$ converge to $u(t_\ast)$.
Since the $x_j$ are bounded, there is a subsequence converging to $x_\ast\in Sing\cap B_0(R)$.
If $x_\ast\in\mathcal{A}_2$, since $u(t_\ast)=e(x_\ast)$, and $u$ is transverse to $e(U\cap\mathcal{A}_2)$ at $t_\ast$, there is $\delta>0$ such that $u|_{[t_\ast-\delta,t_\ast+\delta]}$ does not intersect $e(U\cap\mathcal{A}_2)$, which is a contradiction with the fact that a subsequence of the $x_j$ converge to $x_\ast$.
If $x_\ast\in Sing\setminus\mathcal{A}_2$, the contradiction is with the fact that the image of $u$ does not intersect $\mathcal{N}$.
Note however that it is perfectly possible that the intersection points of $u$ and $e(\mathcal{A}_2\cap B_0(R))$ cluster at the final point $u(l)$.

We have shown that the set of intersection points of $u$ and $e(\mathcal{A}_2\cap B_0(R))$ is discrete $0<t_1<\dots<t_j<\dots$ except at the limit $j\rightarrow\infty$, and is bounded by $l$.

Since $\beta_1(\varepsilon/4)$ is in $int(V_1)$, we can start a lift of $u$ from that point.
Using the argument of equation \ref{bound for the norm of the lifted curve}, we see that the curve $\beta$ will stay in $B_0(R)$.
Thus, the lift will stay in $int(V_1)$ up to $t_1$ since $u|_{[0,t_1)}$ does not intersect $e(Sing)\cap B_0(R)$, so we get a curve $\beta_2:[0,l_2]\rightarrow V_1$ that may end in a point in $\mathcal{A}_2\cup int(V_1)$.

If $\beta_2(l_2)$ is in $\mathcal{A}_2$, we can find a new unequivocal point $\beta_3(0)$ that is linked to $\beta_2(l_2)$ and with $|\beta_3(0)|<|\beta_2(l_2)|$.
Since the point $e(\beta_3(0))=e(\beta_2(l_2))$ belongs to the image of $u$ and is unequivocal, it can only be a nonsingular point, so we can start a new lift of $u|_{[t_2,t_3]}$, and so on.
The claim follows easily.

\end{proof}

\begin{proof}[Proof of Main Theorem \ref{main theorem ambrose}]

We only need to prove that the weak synthesis $M$ built using Theorem \ref{thm: weak synthesis} is complete when $(M_1, p_1)$ is sutured.
It is well known that a local isometry is a covering map when the domain is complete (see for example corollary 2 in \cite{Griffiths Wolf}).

As mentioned in conjecture \ref{ambrose conjecture}, this implies the original Ambrose conjecture when both manifolds are simply connected.

Let $q_{n}$ be a Cauchy sequence in $M$.
Then there is $R>0$ such that $d_M( q_{n} ,q_{1} ) <R$.
Thanks to proposition \ref{d is distance-decreasing}, we can find $x_{n} \in e^{-1} ( q_{n} ) \cap B_{| q_{1} |+R}$. As $x_{n}$ is bounded, we can assume by passing to a subsequence that $x_{n}$ converges to some $x_{0}$, and then $q_{n} \rightarrow e ( x_{0} )$.
\end{proof}

\section{Generic exponential maps}\label{section: generic metric}

A generic perturbation of a Riemannian metric greatly simplifies the types of
singularities that can be found on the exponential map
({\cite{Weinstein}},{\cite{Klok}}) or the cut locus with respect to any point
({\cite{Buchner Stability}}). In {\cite{Weinstein}}, A. Weinstein showed that for a
generic metric, the set of conjugate points in the tangent space near a
singularity of order $k$ is given by the equations:
\[ \text{$\left|\begin{array}{cccc}
     x_{1} & x_{2} & \ldots & x_{k}\\
     x_{2} & x_{k+1} & \ldots & x_{2k-1}\\
     \vdots &  &  & \vdots\\
     x_{k} & x_{2k-1} & \ldots & x_{\frac{k(k+1)}{2}}
   \end{array}\right| =0$} \]
where $x_{1} , \ldots x_{n}$ are coordinates in $\Tone$, and $k(k+1)/2 \leqslant n$.
This is called a conical singularity.

In {\cite{Buchner Stability}}, M. Buchner studied the energy functional on curves starting at $p_{1}$ and the endpoint fixed at a different point of the manifold, as a family of functions parametrized by the endpoint.
The singularities of the exponential map can be detected as degenerate degenerate critical points of the energy functional with both endpoints fixed, so his results also apply to our setting.
He also proved a multitransversality statement about this family of functions that we will comment on later, and then used this information to provide a description of the cut locus of a generic metric.

It is well known that a exponential map only has Lagrangian singularities. 
In {\cite{Klok}}, Fopke Klok showed that the generic singularities of the
exponential maps are the generic singularities of Lagrangian maps. These
singularities are, in turn, described by means of the generalized phase
functions of the singularities. 
This is the approach most useful to our purposes.
We also wish to mention \cite{Janeczko Mostowski} for a different approach and generalizations of some of these results.

\subsection{Generalized phase functions}

A {\emph{generalized phase function}} is a map $F:U \times \RR^{k}
\rightarrow \RR$ such that $D_{q} F= \left( \frac{\partial F}{\partial
q_{1}} , \ldots , \frac{\partial F}{\partial q_{k}} \right) :U \times
\RR^{k} \rightarrow \RR^{k}$ is transverse to $\{0\} \in
\RR^{k}$. We will use a result that relates generalized phase
functions defined at $U \times \RR^{k}$ and Lagrangian subspaces of
$T^{\ast} U$:

\begin{prop}
  If $L \subset T^{\ast} U$ is a Lagrangian submanifold and $p \in L$, it is locally given as the graph of $\phi |_{C} :C \rightarrow T^{\ast} U$, where $C=(D_{ q} F)^{-1} (0)$ and $\phi (x,q)=(x,D_{x} F(x,q) )$, for some
  generalized phase function $F$.
  
  Furthermore, we can assume:
  \begin{itemize}
    \item $k= \tmop{corank} (L,p)$
    
    \item $F(0,0)=0$
    
    \item $0 \in \RR^{k}$ is a critical point of $F(0, \cdot
    ):\RR^{k} \rightarrow \RR$
    
    \item $\frac{\partial^2 F}{\partial q_{i} \partial q_{j}} =0$ for all $i$ and $j$ in $1,\dots,k$
  \end{itemize}
\end{prop}

\begin{proof}
  This is found in section 1 of \cite{Klok}, specifically in proposition 1.2.4 and the comments in page 320 after proposition 1.2.6.
\end{proof}

Given a germ of generalized phase function $F:\RR^{n} \times
\RR^{k} \rightarrow \RR$, the Lagrangian map is built in this
way: $D_{q} F$ is transverse to $\{0\}$, and we can assume the last $k$
$x$-coordinates are such that the derivative of $D_{q} F$ in those coordinates
is an invertible matrix. Let us split the $x$ coordinates in $(y,z) \in
\RR^{n-k} \times \RR^{k}$. Our hypothesis is that $D_{qz} F$
is invertible.

The implicit equations $D_{q} F=0$ defines functions $f_{j} :\RR^{n-k}
\times \RR^{k} \rightarrow \RR$ such that, locally near $0$,
$D_{qz} F(y,f(y,q),q)=0$.

\begin{dfn}
 A \emph{Lagrangian map} $\lambda:L\rightarrow M$ is the composition of a Lagrangian immersion $i:L\rightarrow T^\ast M$ with the projection $\pi:T^\ast  M\rightarrow M$ (a Lagrangian immersion is an immersion such that the image of sufficiently small open sets are Lagrangian submanifolds).
\end{dfn}
\begin{dfn}
 Two Lagrangian maps $\lambda_j=:L_j\rightarrow M_j$, with corresponding immersions $i_j:L\rightarrow T^\ast M$, $j=1,2$, are \emph{Lagrangian equivalent} if and only if there are diffeomorphisms $\sigma:L_1\rightarrow L_2$, $\nu:M_1\rightarrow M_2$  and $\tau:T^\ast M_1\rightarrow T^\ast M_2$ such that the following diagram commutes:
$$
\xymatrix{ L_1 \ar[d]^{\sigma} \ar[r]^{i_1} & T^\ast M_1 \ar[d]^{\tau} \ar[r]^{\pi_1} & M_1 \ar[d]^{\nu}\\
	   L_2 		       \ar[r]^{i_2} & T^\ast M_2 	       \ar[r]^{\pi_2} & M_2
}
$$
and $\tau$ preserves the symplectic structure.
\end{dfn}

Lagrangian equivalence corresponds to equivalence of generalized phase functions (this is proposition 1.2.6 in \cite{Klok}). Two generalized phase functions are equivalent if and only if we can get one from the other composing three operations:
\begin{enumerate}
 \item Add a function $g(x)$ to $F$. This has no effect on the functions
$f_{j}$.

\item Pick up a diffeomorphism $G:\RR^{n} \rightarrow \RR^{n}$,
and replace $F(x,q)$ by $F(G(x),q)$. If the map $G$ has the special form
$G(x)=G(y,z)=(g(y),h(z))$, the effect is to replace the map $(y,q) \rightarrow
(y,f(y,q))$ by $(y,q) \rightarrow (y,h^{-1} (f(g(y),q)))$.

\item Pick up a map $H:\RR^{n} \times \RR^{k} \rightarrow
\RR^{k}$ such that $D_{q} H$ is invertible, and replace $F(x,q)$ by
$F(x,H(x,q))$. If the map $H$ does not depend on the $z$ variables, the effect
is to replace the map $(y,q) \rightarrow (y,f(y,q))$ by $(y,q) \rightarrow
(y,f(y,H(y,q)))$
\end{enumerate}

\subsection{The singularities of a generic exponential map}

Using theorem 1.4.1 in {\cite{Klok}}, we get the following result: fix a smooth
manifold $M$, a point $p \in M$. For a residual set of metrics in $M$ the
exponential map $T_{p} M \rightarrow M$ is nonsingular except at a set
$\tmop{Sing}$, which is a smooth stratified manifold with the following strata
(we describe the different singularities in some detail below):
\begin{itemizedot}
  \item A stratum of codimension $1$ consisting of {\emph{folds}}, or
  Lagrangian singularities of type $A_{2}$.
  
  \item A stratum of codimension $2$ consisting of {\emph{cusps}}, or
  Lagrangian singularities of type $A_{3}$.
  
  \item Strata of codimension $3$ consisting of Lagrangian singularities of
  types $A_{4}$ (swallowtail), $D_{4}^{-}$ (elliptical umbilic) and
  $D_{4}^{+}$ (hyperbolic umbilic).
  
  \item We do not need to worry about the rest, which consists of strata of
  codimension at least $4$.
\end{itemizedot}

\begin{dfn}
  We define the sets $\mathcal{A}_{2}$, $\mathcal{A}_{3}$, etc as the set of all points of $V_{1}$ that have a singularity of type $A_{2}$, $A_{3}$, etc. 
  We also define $\mathcal{C}$ as the set of conjugate (singular) points and $\NC$ as the set of non-conjugate (non-singular) points.
\end{dfn}

Thus, $\tmop{Sing}$ is a smooth hypersurface of $\T$ near a conjugate point of
order $1$ (including $A_{2}$, $A_{3}$ and $A_{4}$ points), and is
diffeomorphic to the product of a cone in $\mathbb{R}^{3}$ with a cube near a
conjugate point of order $2$ (including $D_{4}^{\pm}$). The $A_{2}$ points are
characterized as those for which the kernel of the differential of the
exponential map is a vector line transverse to the tangent plane to
$\tmop{Sing}$.

Furthermore, the image by $\e$ of each stratum of canonical singularities is
also smooth. There might be strata of high codimension that are not uniform,
in the sense that the exponential map at some points in those strata may not
have the same type of singularity (in other words, the singularities are
{\emph{non-determinate}}).
This only happens in some strata of codimension at least $5$, and is not a problem for our arguments.

There are also other generic property that interests us: the image of the
different strata intersect ``transversally'':

Take two different points $x_{1} ,x_{2} \in \T$ mapping to the same point of
$M$, and assume $x_{1}$ and $x_{2}$ lie in $\mathcal{A}_{2} \cup
\mathcal{A}_{3} \cup \mathcal{A}_{4} \cup \mathcal{D}_{4}$. Then the points
$x_{1}$ and $x_{2}$ have neighborhoods $U_{1} ,U_{2}$ such that $\e (U_{1}
\cap \mathcal{C} )$ and $\e (U_{2} \cap \mathcal{C} )$ are transverse (each
pair of strata intersect transversally).
This follows from proposition 1 in page 215 of {\cite{Buchner Stability}}, with $p=2$,
so that $_{2} j_{2}^{k} H( \alpha )$ is transverse to the orbit in
$\mathbb{R}^{2} \times [J^{k}_{ 0} (n,1)]^{2}$ where the first jet is of type
$\mathcal{T}_{1}$ and the second one is of type $\mathcal{T}_{2}$.

For any singularity in the above list, we can choose coordinates near $x$ and
$\e (x)$ so that $\e$ is expressed by standard formulas. For example, the
formulas near an $A_{3}$ point are $(x_{1} , \ldots ,x_{n-1} ,x_{n} )
\rightarrow (x_{1}^{3} \pm x_{1} x_{2} ,x_{2} , \ldots ,x_{n} )$.

The coordinates that we will use are derived using generalized phase functions.
We list the generalized phase functions and the corresponding coordinates for the exponential function that derives from it for the singularities $A_{2}$, $A_{3}$, $A_{4}$ and $D_{4}^{\pm}$:
\addtolength{\leftmargini}{-10pt} %
\begin{itemizedot}\itemsep=10pt %
\small %
  \item $A_{2}$:\begin{tabular}{l}
    $F(x_{1} , \tilde{x}_{1} ,x_{2} ,x_{3} , \ldots ,x_{n} )= \frac{1}{3}
    x_{1}^{3} - \tilde{x}_{1} x_{1}$\\[5pt] %
    $\e : (x_{1} ,x_{2} ,x_{3} , \ldots ,x_{n} ) \rightarrow (x_{1}^{2} ,x_{2}
    ,x_{3} , \ldots ,x_{n} )$
  \end{tabular}
  
  \item $A_{3}$:\begin{tabular}{l}
    $F(x_{1} , \tilde{x}_{1} ,x_{2} ,x_{3} , \ldots ,x_{n} )= \frac{1}{4}
    x_{1}^{4} \pm \frac{1}{2} x_{2} x_{1}^{2} - \widetilde{x_{1}} x_{1}$
    \\[5pt] %
    $\e : (x_{1} ,x_{2} ,x_{3} , \ldots ,x_{n} ) \rightarrow (x_{1}^{3} \pm x_{1}
    x_{2} ,x_{2} ,x_{3} , \ldots ,x_{n} )$
  \end{tabular}

  \item $A_{4}$:\begin{tabular}{l}
    $F(x_{1} , \tilde{x}_{1} ,x_{2} ,x_{3} , \ldots ,x_{n} )= \frac{1}{5}
    x_{1}^{5} + \frac{1}{3} x_{2} x_{1}^{3} + \frac{1}{2} x_{3} x_{1}^{2} -
    \widetilde{x_{1}} x_{1}$\\[5pt] %
    $\e : (x_{1} ,x_{2} ,x_{3} , \ldots ,x_{n} ) \rightarrow (x_{1}^{4}
    +x^{2}_{1} x_{2} +x_{1} x_{3} ,x_{2} ,x_{3} , \ldots ,x_{n} )$
  \end{tabular}
   
  \item $D_{4}^{-}$:\begin{tabular}{l}
    $F(x_{1} ,x_{2} , \tilde{x}_{1} , \widetilde{x_{2}} ,x_{3} , \ldots ,x_{n}
    )= \frac{1}{6} x_{1}^{3} - \frac{1}{2} x_{1} x_{2}^{2} +x_{3} (
    \frac{1}{2} x_{1}^{2} + \frac{1}{2} x_{2}^{2} )- \widetilde{x_{1}} x_{1} -
    \widetilde{x_{2}} x_{2}$\\[5pt] %
    $\e\!:\!(x_{1} ,x_{2} ,x_{3}, \ldots ,x_{n} ) \rightarrow (
    \frac{1}{2} x_{1}^{2} - \frac{1}{2} x_{2}^{2} +x_{1} x_{3} ,-x_{1} x_{2}
    +x_{2} x_{3} ,x_{3}, \ldots ,x_{n} )$
  \end{tabular}
  
  \item $D_{4}^{+}$:\begin{tabular}{l}
    $F(x_{1} ,x_{2} , \tilde{x}_{1} , \widetilde{x_{2}} ,x_{3} , \ldots ,x_{n}
    )= \frac{1}{6} x_{1}^{3} + \frac{1}{6} x_{2}^{3} +x_{1} x_{2} x_{3} -
    \widetilde{x_{1}} x_{1} - \widetilde{x_{2}} x_{2}$\\[5pt] %
    $\e :(x_{1} ,x_{2} ,x_{3} ,x_{4} , \ldots ,x_{n} ) \rightarrow (
    \frac{1}{2} x_{1}^{2} +x_{2} x_{3} , \frac{1}{2} x_{2}^{2} +x_{1} x_{3}
    ,x_{3} ,x_{4} , \ldots ,x_{n} )$
  \end{tabular}
\end{itemizedot}
\addtolength{\leftmargini}{10pt} %

\begin{dfn}\label{dfn: canonical form and adapted coordinates}
 The above expression is the \strong{canonical form} of the exponential map at the singularity.
 The canonical form is only defined for the singularities in the above list.

We call \strong{adapted coordinates} any set of coordinates on $U\subset T_pM$ and $V\supset \e(U)$ for which the
expression of the exponential map is canonical.
\end{dfn}

\begin{dfn}
 Let $U$ be a neighborhood of adapted coordinates near a conjugate point $x$.
 The \strong{lousy metric} on $U$ is the metric whose matrix in adapted coordinates is the identity.
\end{dfn}

\begin{remark}
  We call this metric ``lousy'' because it does not have any geometric meaning, and it depends on the particular choice of adapted coordinates.
 However, it is useful for doing analysis.
\end{remark}

Although the adapted coordinates make the exponential map simple, radial geodesics from $p$ are no longer straight lines, and the spheres of constant radius in $\T$ are also distorted. 
We do not know of any result that gives an explicit canonical formula for the exponential map and also keeps radial geodesics in $\T$ simple. 
The results of section \ref{section: CDCs in adapted coords} suggest that this might be possible to some extent, but the classification that might derive from it must be finer than the one above. 
We will find examples showing that the radial vector can be placed in different, non-equivalent positions.

For example, near an $A_{3}$ point, $\mathcal{C}$ is given by $3x_{1}^{2}
=x_{2}$. The radial vector $r=(r_{1} , \ldots ,r_{n} )$ at $(0, \ldots ,0)$ is
transverse to $\mathcal{C}$, and thus must have $r_{2} \neq 0$. There are two
possibilities:
\begin{itemizedot}
  \item A point is $A_3(I)$ if and only if $r_{2} >0$.
  
  \item A point is $A_3(II)$ if and only if $r_{2} <0$.
\end{itemizedot}
Even though the exponential map has the same expression in both cases (for adequate coordinates), they differ for example in the following:

Let $x \in \mathcal{A}_{3} \cap V_{1}$ (a first conjugate point), and let $U$
be a neighborhood of $x$ of adapted coordinates. Then $\e (V_{1} \cap U$) is
a neighborhood of $\e (x$) if and only if $x$ is $A_3(I)$. A proof for this fact will be
trivial after section \ref{section: A3 are unequivocal}.

In fact, the above can be used as a characterization (for points in $\mathcal{A}_{3} \cap V_{1}$) that shows that the definition is independent of the adapted coordinates chosen. 
We remark that in a sufficiently small neighborhood of an $A_3(I)$ point, there are no $A_3(II)$ points, and viceversa.

We will get back to this distinction later, and we will also make a similar distinction with $D_{4}^{+}$ points.

\begin{remark}
  In the literature, it is common to see singularities of real functions of type $A_{3}$ further subdivided into $A_{3}^{+}$ and $A_{3}^{-}$ points. A canonical form for an
  $A_{3}^{\pm}$ singularity is 
  $$F^{\pm} (x_{1} , \tilde{x}_{1} ,x_{2} ,x_{3} ,
  \ldots ,x_{n} )= \pm \frac{1}{4} x_{1}^{4} - \frac{1}{2} x_{2} x_{1}^{2} -
  \widetilde{x_{1}} x_{1}$$
  When $F^{\pm}$ are generalized phase functions,
  both subtypes give equivalent singularities. However, in the work of
  Buchner, the same singularities appear, now as the energy function in a
  finite dimensional approximation to the space of paths with fixed endpoints.
  In this second context, it is not equivalent if a geodesic is a local
  minimum, or a maximum, of the energy functional, and it would make sense to
  use the distinction between $A_{3}^{+}$ and $A_{3}^{-}$, rather than the
  similar-but-not-the-same distinction between $A_3(I)$ and $A_3(II)$.
  
  This can also serve as an illustration that the classifications of singularities of the exponential map by F. Klok and M. Buchner are not equivalent, even though the final result is indeed quite similar. In the classification of F. Klok, the $A_3$ singularities are not divided into the two subclasses $A_3^+$ and $A_3^-$.
\end{remark}

\begin{dfn}\label{the set of generic metrics}
We define $\mathcal{H}_{M}$ as the set of Riemannian metrics for the smooth manifold $M$ such that the singular set of $\e$ is stratified by singularities of types $A_{2}$, $A_{3}$, $A_{4}$ and $D_{4}^{\pm}$ with the codimensions listed above, plus strata of different types with codimension at least $4$, and such that the images of any two strata intersect transversally as described above.
\end{dfn}

\begin{theorem}\label{thm: H_M residual}
 $\mathcal{H}_{M}$ is residual in the set of all Riemannian metrics on $M$.
\end{theorem}

\begin{proof}
This is the work of M. Buchner and F. Klok, as we have shown in this section.
\end{proof}

\section{Proof of Theorem \ref{main theorem generic}}\label{section: proof of main theorem generic}

In the previous section we have classified the points of $T_pM$ for a generic Riemannian manifold according to the singularity of the exponential map at that point.
We use that classification to split $T_pM$ into two sets, according to the role that they play when proving that the manifold is sutured:
\begin{dfn}
$$\mathcal{I} =( \NC \cup \mathcal{A}_{3}(I)  ) \cap V_{1}$$
$$\mathcal{J}=(\mathcal{A}_{2} \cup \mathcal{A}_{3}(II) \cup \mathcal{A}_{4}   \cup \mathcal{D}_{4}^{\pm} ) \cap V_{1}$$
\end{dfn}

\begin{theorem}\label{claim: hypothesis of the synthesis theorem}
Points in $\mathcal{I}$ are unequivocal, and any point in $\mathcal{J}$ is strongly linked to a point in $\mathcal{I}$ of smaller radius.
\end{theorem}

\begin{proof}[Proof of Main Theorem \ref{main theorem generic}]
By definition, $V_1=\mathcal{I}\cup \mathcal{J}$ for a metric in $\mathcal{H}_M$.
By Theorem \ref{claim: hypothesis of the synthesis theorem}, all metrics on $\mathcal{H}_M$ are sutured.
Main Theorem \ref{main theorem generic} follows by application of Theorem \ref{thm: H_M residual}.
\end{proof}

\subsection{$A_3(I)$ first conjugate points are unequivocal}\label{section: A3 are unequivocal}

\begin{lem}\label{A3I are unequivocal}
 Any $x\in V_1$ of type $A_3(I)$ is unequivocal.
\end{lem}

\begin{proof}
Consider an $A_3(I)$ point $x$ in a manifold $(M_{1} ,p_{1} )$ whose curvature is $L$-related to the curvature of $(M_{2} ,p_{2} )$, and use adapted coordinates near $x=(0,0,0)$, in an arbitrarily small neighborhood $O$:
\begin{itemizedot}
  \item Define $\gamma (x_{1} ,x_{3} )=x_{1}^{2}$.
  
  \item Let $A$ be the subset of $O$ given by $x_{2} < \gamma (x_{1} ,x_{3}
  )$. $ e_{1}$ maps diffeomorphically $A$ onto a big subset of $ e_{1} (O)$. Only
  the points with $x_{1} =0,x_{2} \geqslant 0$ are missing. $x$ is $A_3(I)$, so
  $\bar{A} \subset V_{1}$, and $ e_{1} (O \cap V_{1} )$ is open.
  
  \item For any $(x_{1} ,x_{3} )$, the pair of points $(x_{1} ,x_{1}^{2}
  ,x_{3} )$ and $(-x_{1} ,x_{1}^{2} ,x_{3} )$ map to the same point by
  $ e_{1}$, the curve $t \rightarrow (t,t^{2} ,x_{3} )$, $t \in [-x_{1} ,x_{1}
  ]$ maps to a tree-formed curve. This shows that the two points map to the
  same point by $ e_{2}$ as well.
  
  \item Define a map $\varphi :  e_{1} (O) \rightarrow  e_{2} (O)$ by $\varphi
  (p)=  e_{2} (a)$, for any $a \in \bar{A}$ such that $p=  e_{1} (a)$. By the
  above, this is unambiguous.
  
  \item For a pair of linked points $x=(x_{1} ,x_{1}^{2} ,x_{3} , \ldots ,x_{n} )$ and $\bar{x} =(-x_{1} ,x_{1}^{2} ,x_{3} , \ldots ,x_{n} )$, we have two different local isometries from a neighborhood of $p=  e_{1} (x)=  e_{1} (  \bar{x} )$ into $M_{2}$, given by $ e_{2} \circ (  e_{1} |_{O_{i}} )^{-1}$, for neighborhoods $O_{i}$ of $x$ and $\bar{x}$ such that $ e_{1} (O_{1} )= e_{1} (O_{2} )$ and we need to show that they agree. 
  They both send $p$ to the same point, and we only need to check that their differential at $p$ is the same.
  These are the linear isometries $I_x$ and $I_{\bar{x}}$, and they agree by \ref{lem: si e_1 o Y es tree formed, sus extremos tienen la misma I_Y(.)}.

  \item We know that $\varphi \circ  e_{1} (x)=  e_{2} (x)$, for $x \in
  \bar{A}$. Let $y \in O \setminus \bar{A}$. There is a unique point $x$
  in the radial line through $y$ in $\partial A$. We know $\varphi \circ  e_{1}
  (x)=  e_{2} (x)$, and the radial segment from $x$ to $y$ map by both $\varphi
  \circ  e_{1}$ and $ e_{2}$ to a geodesic segment with the same length,
  starting point and initial vector. We conclude $\varphi \circ  e_{1} (y)=
   e_{2} (y)$.
\end{itemizedot} 
\end{proof}

\begin{remark}
  The only place where we used that the point is $A_3(I)$ is when we assumed
  that $A \subset V_{1}$.
\end{remark}

\subsection{Conjugate flow}\label{section: conjugate flow}

We now introduce the main ingredient in the construction of the linking curves.
The idea in the definition of conjugate flow was used in lemma 2.2 of \cite{Hebda82} for a different purpose.

Near a conjugate point of order 1, the set $C$ of conjugate points is a smooth hypersurface.
Furthermore, we know $\ker  d F$ does not contain $r$ by Gauss' lemma.
Thus we can define a one dimensional distribution $D$ within the set of points of order $1$ by the rule:
\begin{equation}
  \label{1d distribution at conjugate points} D= ( \ker  d F \oplus <r> ) \cap
  T  C
\end{equation}
\begin{dfn}
  \label{definition: conjugate flow}A {\strong{conjugate descending curve}}
  (CDC) is a smooth curve, consisting only of $A_{2}$ points, except possibly
  at the endpoints, and such that the speed vector to the curve is in $D$ and
  has negative scalar product with the radial vector $r$. Therefore, the
  radius is decreasing along a CDC.
  
  The {\strong{canonical parametrization}} of a CDC $\gamma$ is the one that
  makes $d\e(\gamma')$ a unit vector. By Gauss lemma, it is also the one
  that makes \mbox{$dR( \gamma' )\!=\!1$}.
\end{dfn}

\begin{dfn}
  \label{definition: retort}Let $\alpha :[0,t_{1} ] \rightarrow T_{p} M$ be a
  smooth curve, and $x \in \T$ be a point such that $\e (x)= \e ( \alpha
  (t_{1} ))$. A curve $\beta :[0,t_{1} ] \rightarrow T_{p} M$ is a
  {\strong{retort}} of $\alpha$ starting at $x$ if and only if $\alpha (t) \neq
  \beta (t_{1} -t)$ for any $t \in [0,t_{1} )$, but $\e ( \alpha (t))= \e ( \beta
  (t_{1} -t))$ for any $t \in [0,t_{1} ]$, and $\beta (t)$ is NC for any $t
  \in (0,t_{1} )$. Whenever $\beta$ is a retort of $\alpha$, we say that
  $\beta$ {\strong{replies}} to $\alpha$. A {\strong{partial retort}} of
  $\alpha$ is a retort of the restriction of $\alpha$ to a subinterval $[t_{0}
  ,t_{1} ]$, for $0<t_{0} <t_{1}$.
\end{dfn}

We have seen that near an $A_{2}$ point $x$, there are coordinates near $x$
and $\e (x)$ such that $\e$ reads $(x_{1} ,x_{2} , \ldots ,x_{n} ) \rightarrow
(x_{1}^{2} ,x_{2} , \ldots ,x_{n} )$. The $A_{2}$ points are given by $x_{1}
=0$, and no other point $y \neq x$ maps to $\e (x)$. Thus, there is a
neighborhood $U$ of any CDC such that any CDC contained in $U$ has no non-trivial retorts contained in $U$.

\begin{lem}
  \label{unbeatable lemma}Let $x$ be an $A_{2}$ point. Then there is a
  $C^{\infty}$ CDC $\alpha :[0,t_{0} ) \rightarrow \T$ with $\alpha (0)=x$.
  The CDC is unique, up to reparametrization. Furthermore:
  \begin{itemizedot}
    \item $| \alpha (0)|-| \alpha (t_{0} )|= \tmop{length} ( \e \circ \alpha
    )$
    
    \item If $\beta$ is a non-trivial retort of $\alpha$, then of course,
    $$\tmop{length} ( \e \circ \alpha )= \tmop{length} ( \e \circ \beta ),\,
    \text{but}\, | \beta (t_{0} )|-| \beta (0)|< \tmop{length} ( \e \circ
\beta ).$$ 
We say that segments of descending conjugate flow are {\strong{unbeatable}}.
  \end{itemizedot}
\end{lem}

\begin{proof}
  Both $\mathcal{A}_{2}$ and the distribution $D$ are smooth near $x$, so the
  first part is standard.
  
  We also compute:
  \[ \tmop{length} ( \e \circ \alpha ) = \int |( \e \circ \alpha )' |= \int |d
     \e ( \alpha' )| \]
  By definition of $D$, $\alpha' =ar+v$ is a linear combination of a multiple
  of the radial vector and a vector $v \in \ker (d \e )$. By the Gauss lemma,
  $|d \e ( \alpha' )|=a$. On the other hand, $v$ is tangent to the spheres of
  constant radius, so:
  \[ | \alpha (0)|-| \alpha (t_{0} )|= \int \frac{d}{dt} | \alpha |= \int a=
     \tmop{length} ( \e \circ \alpha ) \]
  For a retort $\beta :[0,t_{1} ] \rightarrow \T$, we also have $\beta' =b r+v$
  for a function $b:[0,t_{1} ] \rightarrow \mathbb{R}$ and a vector $v(t) \in
  T_{\beta (t)} ( \T )$ that is always tangent to the spheres of constant
  radius, and $v(t)$ is not identically zero because $ e_{1} \circ \beta$ is
  not a geodesic. However, $\beta (s)$ is non-conjugate, so $|d \e ( \beta'
  )|= \sqrt{b^{2} +|d \e (v)|^{2}} >b$. The result follows.
\end{proof}

\begin{remark}
 We recall that the plan is to build linking curves, whose composition with the exponential is tree formed.
 If a linking curve contains a CDC, it must also contain a retort for that CDC.
 The ``unbeatable'' property of CDCs is interesting, because the radius decreases along a CDC and along the retort it never increases as much as it decreased in the first place.
\end{remark}

\subsection{CDCs in adapted coordinates near $A_{3}$ points\label{subsection: CDCs in adapted coords near A3}}

As we mentioned in section \ref{section: generic metric}, the radial vector field, and the spheres of constant radius of $\T$, that have very simple expressions in standard linear coordinates in $\T$, are distorted in canonical coordinates.
Thus, the distribution $D$ and the CDCs do not always have the same expression in adapted coordinates.
In this section, we study CDCs near an $A_{3}$ point.
We will use the name $R: \T\rightarrow \mathbb{R}$ for the radius function, and $r$ for the radial vector field, and we assume that our conjugate point is a first conjugate point (it lies in $\partial V_{1}$).

In a neighborhood $O$ of special coordinates of an $A_{3}$ point,
$\mathcal{C}$ is given by $3x_{1}^{2} =x_{2}$. At each $A_{3}$ point, the
kernel is spanned by $\frac{\partial}{\partial x_{1}}$. At points in
$\mathcal{C}$, we can define a 2D distribution $D_{2}$, spanned by $r$ and
$\text{$\frac{\partial}{\partial x_{1}}$}$. We extend this distribution to all
of $O$ in the following way:

\begin{dfn}
For any point $x \in O$, there are $y \in   \mathcal{C}$ and $t_{0}$ such that $x= \phi_{t_{0}} ( y )$, where $\phi_{t}$ is the radial flow, and $y$ and $t$ are unique. Define $D_{2} ( x )$ as $( \phi_{t_{0}} )_{\ast} ( D_{2} ( y ) )$.
\end{dfn}

The reader may check that $D_2$ is integrable.
Let $P$ be the integral manifold of $D_{2}$ through $x_{0} =(0, 0 ,0)$.
We can assume $P$ is a graph over the $x_1,x_2$ plane: $x_3=p(x_1,x_2)$.
$\mathcal{A}_{3}$ is transverse to $D_{2}$, so $\{x_{0} \}= \mathcal{A}_{3} \cap P$.
The integral curve $C$ of $D$ through $x_{0}$ is contained in $P$, and $C\setminus \{ x_{0} \}$ consists of two CDCs.
We claim that if the point is $A_3(I)$, the two CDCs descend into $x_{0}$, but if the point is $A_3(II)$, they
start at $x_{0}$ and flow out of $O$.
$P$ is also obtained by flowing the CDC with the radial vector field.

We can assume that $r$ is close to $r(x_{0} )$ in $O$. The tangent $T_{x}$ to
the sphere of constant radius $\{y:R(y)=R(x)\}$ must contain
$\frac{\partial}{\partial x_{1}}$ (the kernel of $d \e$) if $x \in
\mathcal{C}$, by Gauss lemma, and we can assume that the angle between $T_{x}$
and $\frac{\partial}{\partial x_{1}}$ is small if $x \nin \mathcal{C}$.

The curves $\{R(x)=R_1\}\cap P$, for any $R_1$, are all smooth graphs over the $x_1$ axis.
We claim that the curve $\{R(x)=R(x_0)\}\cap P$ may not intersect $3x_{1}^{2} < x_{2}$.
Assume that $R(y)=R(x_0)$ for some $y=(y_1,y_2)$ with $y_1<0$ and $y_2>3y_1^2$.
Then there is a curve $\{R(x)=R(x_0)-\varepsilon\}\cap P$, for some $0<\varepsilon<<1$, must intersect $\mathcal{C}$ at a point $(x_1,3x_1^2,p(x_1,3x_1^2))$ with $x_1<0$, and the tangent to $\{R(x)=R_1\}\cap P$ must be $\frac{\partial}{\partial x_{1}}$.
Taking coordinates, $x_1,x_2$ in $P$, we see it is not possible that a graph $(x_1,t(x_1))$ over the $x_1$ axis has $t(0)<0$, $t(y_1)>3y_1^2$, and intersect the curve $\mathcal{C}\cap P$ only with horizontal speed.

It follows that $x_0$ is a local maximum, or minimum, of $R$, within $\mathcal{C}$.
If $r_{2} >0$ ($A_3(I)$ points), then $R(x)\geq R(x_{0}$) for $x \in \mathcal{C}$, while $r_{2} <0$ ($A_3(II)$ points), implies $R(x)\leq R(x_{0}$) for $x \in \mathcal{C}$.

Thus, $A_3(I)$ points are {\emph{terminal}} for the conjugate flow, but $A_3(II)$ points are not.
We have proved the following:

\begin{lem}\label{lemma: CDCs near A3 points}
 In a neighborhood $O$ of adapted coordinates near an $A_3$ point $x_0$:
 \begin{itemize}
  \item $\mathcal{C}$ is foliated by integral curves of $D$.
  \item $\mathcal{A}_2$ is foliated by CDCs.
  \item If $x_0$ is $A_3(I)$, exactly two CDCs in $O$ flow into each $A_3$ point. If $x_0$ is $A_3(II)$, exactly two CDCs in $O$ flow out of each $A_3$ point.
  \item If $x_0$ is $A_3(I)$, every CDC in $O$ flow into some $A_3$ point. If $x_0$ is $A_3(II)$, every CDC in $O$ flow out of some $A_3$ point.
 \end{itemize}
\end{lem}

\subsection{$A_{3}$ joins\label{section: A3 joins}}

We can continue a CDC as long as it stays within a stratum of $A_{2}$ points.
As we have seen, a CDC may enter a different singularity. The most important
situation is when the CDC reaches an $A_{3}$ point, because then we can start
a non-trivial retort right after the CDC.

The set of conjugate points is a graph over the $x_{1} ,x_{3}$ plane: $x_{2} = \alpha (x_{1} ,x_{3} )=3x_{1}^{2}$.
A CDC is written $t \rightarrow (t,3t^{2},x_{3} (t))$, for $t \in [t_{0} ,0]$, finishing at an $A_{3}$ point $(0,0,x_{3} (0))$.
We can start a retort for this segment of CDC starting at the $A_{3}$ point.
The retort for this CDC is given explicitly by $t \rightarrow (-2t,3t^{2} ,x_{3}(t))$.

These curves, composed of a segment of CDC plus the corresponding retort, map to a fully tree-formed map that shows that the point $(t,3t^{2} ,x_{3} )$ is linked to $(-2t,3t^{2} ,x_{3} )$.
We say that the CDC and the retort given above are joined with an {\strong{$A_{3}$ join}}.

\subsection{Avoiding some obstacles}\label{section:avoiding some obstacles}

In order to build linking curves, it is simpler to replace CDCs with curves that are close to CDC curves, but avoid certain ``obstacles''. The following remark helps in that respect:

{\emph{A curve that is sufficiently $C^{1}$-close to a
CDC is also unbeatable}}. Actually, we can say more: the greater the angle
between $r_{x}$ and $\ker d_{x}  e_{1}$, the more we can depart from the CDC.

\begin{dfn}
  The \strong{slack} $A_{x}$ at a first order conjugate point $x$ is the absolute value of the sine of the angle between $D_{x}$ and $\ker  (d_{x}  e_{1} )$.
\end{dfn}

\begin{remark}
  The slack is positive if and only if the point is $A_{2}$
\end{remark}

\begin{lem}
  \label{slack lemma}For any positive numbers $R>0$ and $a>0$ there are constants $c>0$ and $\varepsilon>0$ depending
  on $M$, $R$ and $a$ such that the following holds:

  If a smooth curve $\alpha :[0, T ] \rightarrow \Tone$ of $A_{2}$ points satisfies the following properties:
  \begin{enumerate}
   \item $| \alpha (0)| \leqslant R$
   \item $\langle \alpha'(t), r_{\alpha(t)}\rangle <0$ for all $t\in[0,T]$
   \item $A_{\alpha(t)}>a$  for all $t\in[0,T]$
   \item $\alpha' (t)$ is within a cone around $D$ of amplitude $c$ for all $t\in[0,T]$
  \end{enumerate}
  then it holds that:
  \begin{itemize}
   \item $\alpha$ is $\varepsilon$-unbeatable: any retort $\beta$ satisfies  
   $$| \beta (T )|-| \beta (0)|<\left(| \alpha (0)|-| \alpha (T )| \right)(1 - \varepsilon)$$
   \end{itemize}
\end{lem}

\begin{proof}
  Fix a neighborhood $U$ of adapted $A_2$ coordinates that contains the image of $\alpha$.
  We assume that one such $U$ contains all of the image of $\alpha$, otherwise we  split $\alpha$ into parts.
  
  Let $v(t)$ be the vector at $\alpha(t)$ such that $d(t) = -r(\alpha(t)) + v(t)$ belongs to $D_x$.
  Then the slack $A_{\alpha (t)}$ is $\frac{|r(\alpha (t))|}{|-r(\alpha (t))+v(t)|} = \frac{1}{|d(t)|}$.
  
  We assume that $\alpha$ has the canonical parametrization, so that $\alpha'(t)= d(t)+p(t)$, with $p(t)$ is orthogonal to $d$.
  Then $|p(t)|<c |d(t)|=\frac{c}{A_{\alpha(t)}}<\frac{c}{a}$.

   We compute %
  \[ | \alpha (T)|-| \alpha (0)|= \int_{0}^{T} \frac{d}{dt} | \alpha
     |=  \int_{0}^{T} \langle r, \alpha' \rangle  \leq T(-1+\frac{c}{a}) \]
  
  An $A_2$ point only has one preimage in $U$, so any retort $\beta$ of $\alpha$ lies outside of $U$. 
  As $\e ( \alpha (t))= \e ( \beta (T -t))$, we have:
  $$|d \e (r(\alpha(t))) - d\e(r(\beta(T-t)))|>\varepsilon_1$$
  for some $\varepsilon_1$ depending on $U$.
  $U$ contains a ball around $x$ of radius at least $r_0$, a number which depends on $a$, $R$, and an upper bound on the differential of the slack in the ball of radius $R$.
  Thus, we can switch to a smaller $\varepsilon_1>0$ that depends only on $a$
and $R$.
  
  Write $\beta'(t)=b(t)r(\beta (t)) +w(t)$, where $w$ is a vector orthogonal to $r(\beta ( t ))$.
  It follows from the above that $|b(t)|<1-\varepsilon_2 $ for some
$\varepsilon_2 $ depending on $\varepsilon_1$ and a lower bound for the norm of the differential of $x\rightarrow (\e(x), d\e(x))$ in the ball or radius $R$ in $T_p M$.
  
  We compute:
  \[ | \beta (T)|-| \beta (0)|= \int_{0}^{T} \frac{d}{dt} | \beta
     |=  \int_{0}^{T} b(t)  \leq T(1-\varepsilon_2) \]

  and if $c<a\varepsilon_2/2$, we get
  \[ | \beta (T)|-| \beta (0)|\leq T(1-\frac{c}{a}) - T\varepsilon_2/2 \leq | \alpha (0)|-| \alpha (T)|- T\varepsilon_2/2.
  \]
  We are using the canonical parametrization for $\alpha$, so $T$ is the length of the curve $\exp\circ\alpha$.
  By lemma \ref{unbeatable lemma}, $T$ is also $| \alpha (0)|-| \alpha (T)|$.
  The result follows with $\varepsilon = \varepsilon_2/2$.

\end{proof}

With this lemma, we can perturb a CDC slightly to avoid
some points:

\begin{dfn}
  An {\emph{approximately conjugate descending curve}} (ACDC) is a $C^1$ curve $\alpha$ of $A_{2}$ points such that $\alpha' (t$) is within a cone around $D$ of amplitude $c$, where $c$ is the constant in the previous lemma for $R= |\alpha (0)|$.
\end{dfn}

\subsection{Conjugate Locus Linking Curves}

Let us assume that we have an ACDC $\alpha :[0,t_{0} ] \rightarrow \T$ starting at a point $x \in \mathcal{J}$, whose interior consists only of $A_{2}$ points and ending up in an $A_{3}$ point.
We know that we can start a retort $\widetilde{\alpha_{1}}$ at the $A_{3}$ point.

We can continue the retort while it remains in the interior of $V_{1}$, where $ e_{1}$ is a local diffeomorphism and we can lift any curve.
However, we might be unable to continue the retort up to $x$ if the returning curve hits the set of conjugate points.

If we hit an $A_{2}$ point $y= \widetilde{\alpha_{1}} (t_{1} )$, we can take
a ACDC $\beta :[0,t_{2} ] \rightarrow V_{1}$ starting at this point and ending
in an $A_{3}$ point. If $\beta$ has a retort $\tilde{\beta} :[0,t_{2} ]
\rightarrow \T$ that ends up in a non-conjugate point $\tilde{\beta} (t_{2}
)$, we can continue with the retort $\widetilde{\alpha_{2}}$ of $\alpha
|_{[0,t_{0} -t_{1} ]}$ starting at $\tilde{\beta} (t_{2} )$. If
$\widetilde{\alpha_{2}}$ can be continued up to $x= \alpha (0)$, the concatenation of $\alpha$, $\widetilde{\alpha_{1}}$, $\beta$, $\tilde{\beta}$ and $\widetilde{\alpha_{2}}$ is a linking curve (see figure \ref{figure: standard T}).

There are a few things that may go wrong with the above argument: the retort
$\widetilde{\alpha_{1}}$ may meet $\mathcal{J} \setminus \mathcal{A}_{2}$, or
$\beta$ may not admit a full retort starting at $\beta (t_{2} )$, or $\alpha
|_{[0,t_{0} -t_{1} ]}$ may not admit a full retort starting at $\tilde{\beta}
(t_{2} )$. The first problem can be avoided if the ACDCs are built to dodge some small sets, as we will see later.
Then, if we assume that a retort never meets $\mathcal{J} \setminus \mathcal{A}_{2}$, we can iterate the above argument whenever a retort is interrupted upon reaching an $\mathcal{A}_{2}$ point.
We will prove later that the argument only needs to be applied a finite number of times.

This is the motivation for the following definitions:

\begin{dfn}\label{definition: conjugate linking curve}
  A {\strong{finite conjugate linking curve}} (or FCLC, for short) is a continuous \emph{linking curve} $\alpha :[0,t_{0} ] \rightarrow T_{p} M$ that is the concatenation $\alpha = \alpha_{1} \ast \ldots \ast \alpha_{ n}$ of ACDCs and non-trivial retorts of those ACDCs,  all of them {\emph{of finite length}}.%
\end{dfn}

We will build the FCLCs in an iterative way, as hinted at the beginning of this section, by concatenation of ACDCs and retorts of those ACDCs.

\begin{dfn}
  An {\strong{aspirant curve}} is an absolutely continuous curve $\alpha :[0,t_{0} ]\rightarrow T_{p} M$ that is the concatenation $\alpha = \alpha_{1} \ast\ldots \ast \alpha_{ n}$ of ACDCs and non-trivial retorts of those ACDCs,
  such that:
  \begin{itemizedot}
    \item Starting with the tuple $( \alpha_{1} , \ldots , \alpha_{n} )$ consisting of the curves that $\alpha$ is made of in the same order, we can reach a tuple \strong{with no retorts}, by iteration of the
    following rule:
    
    {\emph{Cancel an ACDC together with a retort of that ACDC that follows
    right after it: $( \alpha_{1} , \ldots , \alpha_{j-1} , \alpha_{j} ,
    \alpha_{j+1} , \alpha_{j+2} , \ldots \alpha_{n} ) \rightarrow ( \alpha_{1}
    , \ldots , \alpha_{j-1} , \alpha_{j+2} , \ldots \alpha_{n} )$, if
    $\alpha_{j+1}$ is a retort of $\alpha_{j}$. }}
    
    \item The extremal points of the $\alpha_{i}$ are called the
    {\strong{vertices}} of $\alpha$. The vertices of $\alpha$ fall into one
    of the following categories:
    \begin{itemizeminus}
      \item starting point (first point of $\alpha_{1}$): a point in
      $\mathcal{J}$.
      
      \item end point (last point of $\alpha_{n}$): a point in $\mathcal{I}$.
      
      \item $A_{3}$ join, as explained in section \ref{section: A3 joins}.
      
      \item a {\strong{splitter}}: a vertex that joins two ACDCs whose
      concatenation is also a ACDC.
      
      \item a {\strong{hit}}: a vertex that joins a retort that reaches
      $\mathcal{A}_{3}(I) $ transversally, and an ACDC starting at the intersection point.
      
      \item a {\strong{reprise}}: a vertex that joins a retort that
      completes its task of replying to a ACDC $\alpha_{j}$, and the retort
      for a different ACDC $\alpha_{i}$ (it follows from the first condition
      that $i<j$).
  \end{itemizeminus}
\end{itemizedot}

  The \strong{tip} of $alpha$ is its endpoint $\alpha(t_0)$.

  The {\strong{loose}} ACDCs in $\alpha = \alpha_{1} \ast \ldots \ast
  \alpha_{k}$ are the ACDC curves $\alpha_{j}$ for which there is no retort in
  $\alpha$.
  
  An aspirant curve is \strong{saturated} if there are no loose ACDCs.
\end{dfn}

The three new types of vertices: \emph{splitters}, \emph{hits} and \emph{reprises}, always come in packs.
We have already shown one example of how they could appear, but we formalize that construction in the following definition.

\begin{dfn}
  A {\strong{standard T}} consists of three vertices: a splitter, a hit and
  a reprise, such that the six curves $\alpha_{i}$ contiguous to the
  three points map to a $T$-shaped curve, with two curves mapping into each
  segment of the T (see figure \ref{figure: standard T}).
\end{dfn}

\begin{figure}[H]
  \includegraphics[width=0.8\textwidth]{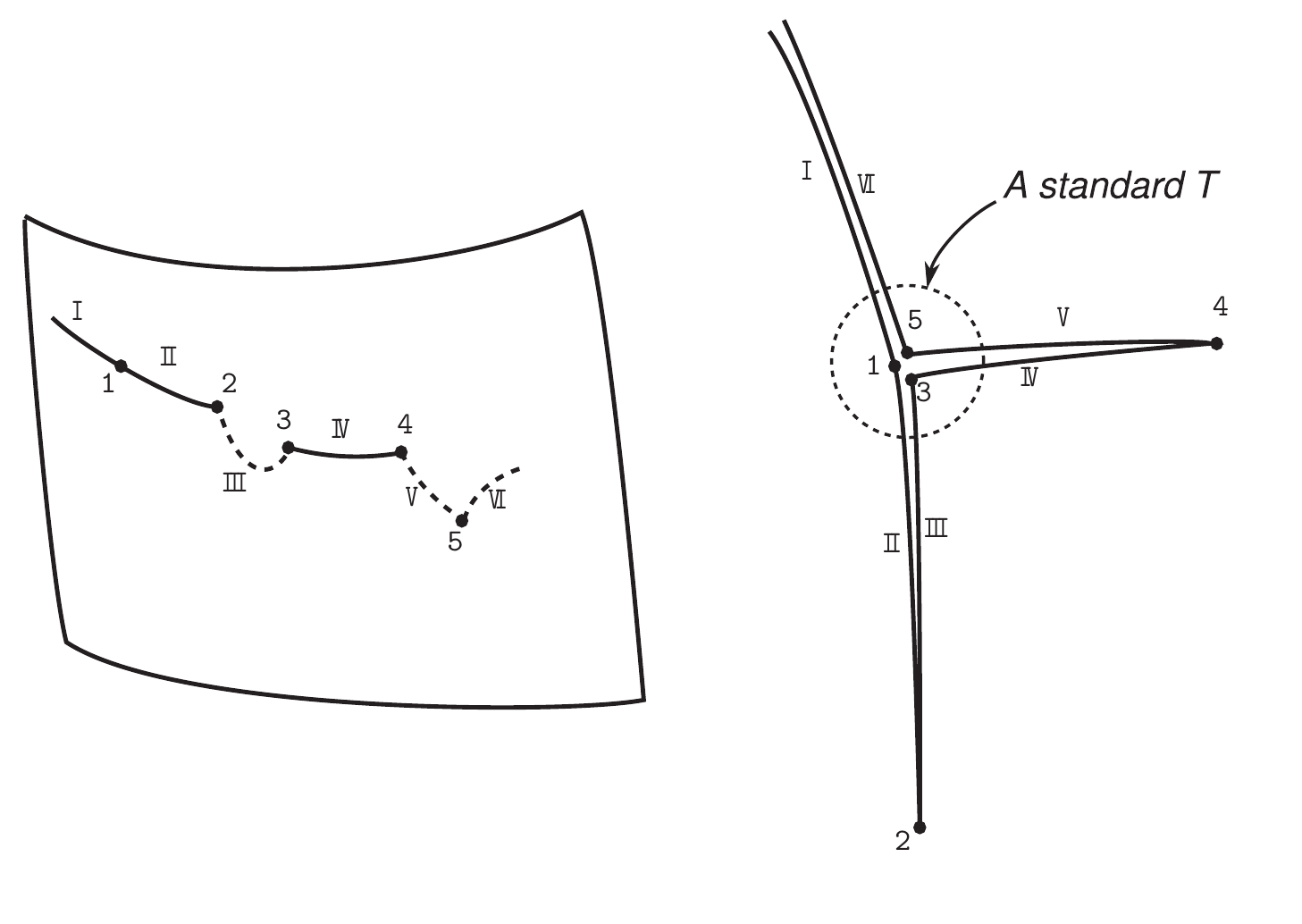}
  \caption{A {\emph{Standard T}}: The left hand side
  displays a curve $\alpha$ in $\T$, while the right hand side displays $\e
  \circ \alpha$.
  I, II and IV are ACDCs, III is the retort of II, V is the retort of IV, and
  VI is the retort of I. Vertices 2 and 4 are $A_{3}$ joins, vertex 1 is a
  splitter, vertex 3 is a hit and vertex 5 is a reprise. There can be more
  than two segments between a splitter and its matching hit, and between a hit
  and its matching reprise.}
  \label{figure: standard T}
\end{figure}

\begin{prop}\label{main properties of linking curves}
  Let $\alpha=\alpha_1\ast\dots\ast\alpha_k$ be a saturated aspirant curve between $x,y \in \Tone$. Then:
  \begin{itemizedot}
    \item $|x|>|y|$
    
    \item $\alpha$ is an FCLC.
  \end{itemizedot}
\end{prop}

\begin{proof}
  The first part follows trivially from lemma \ref{unbeatable lemma} and its generalization, lemma \ref{slack lemma}. Each pair of a ACDC and its retort adds a negative amount to $|\alpha(0)|=|x|$.
  
  For the second part, let $x\sim y$ and $x\sim y$ whenever $x=\alpha_i(s)$ lies on an ACDC $\alpha_i$ defined on $[0,l_i]$ and $y=\alpha_j(l_i-s)$ on its retort $\alpha_j$ defined on the same interval, so that $\alpha_i(t)=\alpha_j(l_i-t)$.
  We also identify the triples of vertices that belong to each standard T.
  Let $T:[0,l]\rightarrow \Gamma$ be the identification map associated to the relation $\sim$.

  We must show that $u= \exp \circ \alpha$ is tree-formed with respect to $T$: 
  let $t_{1}$, $t_{2}$ such that $T(t_{1} )=T(t_{2} )$, and $\varphi$ a continuous $1$-form along $u$ $(\varphi (s) \in T^{\ast}_{u(s)} M)$ that factors through $\Gamma$. 
  Then we claim that:
  \begin{equation}
    \int_{t_{1}}^{t_{2}} \varphi (s)(u' (s))ds \label{integral that must be 0
    for the curve to be tree formed}
  \end{equation}
  splits as a sum of integrals over the image by $\exp$ of an ACDC and the image of its matching retort.
  The curves in each such pair have the same image, and the integrals cancel out, as the integral of a $1$-form is
  independent of the parametrization, and only differs by sign.
  
  Suppose first that $t_{1}$ is in the domain of an ACDC $\alpha_{i}$ and $t_{2}$ lies in the retort $\alpha_{j}$ of $\alpha_{i}$.
  We recall it is possible to reach an empty tuple by canceling adjacent pairs of an ACDC and its retort.
  Thus, in order to cancel $\alpha_{i}$ and $\alpha_{j}$, it must be possible to cancel all the curves $\alpha_{k}$ with $i<k<j$.
  These curves can be matched in pairs $\{ ( \alpha_{n} , \alpha_{m} ) \}_{( n,m ) \in \mathcal{P}}$ of ACDC and retort, with $i<n<m<j$ for each pair $( n,m ) \in \mathcal{P}$.
  Then we have:
\begin{align*}
    \int_{t_{1}}^{t_{2}} \varphi (s)(u' (s))ds  =\, & 
    \int_{t_{1}}^{t_{2}^{i}}
      \varphi (s)(( \exp \circ \alpha_{i} )' (s))ds+
    \\
    & \sum_{( n,m ) \in \mathcal{P}} \Big( \int_{t^{n}_{1}}^{t^{n}_{2}}
    \varphi (s)(( \exp \circ \alpha_{n} )' (s))ds\:+\\
    & \hskip46pt\int_{t^{m}_{1}}^{t^{m}_{2}} \varphi (s)(( \exp \circ \alpha_{m}
)' (s))ds \:
    \Big) +\\
    & \int_{t_{1}^{j}}^{t_{2}} \varphi (s)(( \exp \circ \alpha_{j} )' (s))ds
\end{align*}

  The remaining two integrals also cancel out, proving the claim.
  
  If $t_{1}$ and $t_{2}$ are two of the three points of a standard T, we can
  take points $t^{\ast}_{1}$ and $t^{\ast}_{2}$ as close to $t_{1}$ and
  $t_{2}$ as we want, but in an ACDC and its retort, respectively, and such that $T(t_1^\ast)=T(t_2^\ast)$.
  The result follows because the integral \ref{integral that must be 0 for the curve to
  be tree formed} depends continuously on $t_{1}$ and $t_{2}$.

\end{proof}

\subsection{Existence of FCLCs}

The goal of this section is to prove the existence of an FCLC starting at an arbitrary point $x \in \mathcal{J}$.
The set $\{y:|y|<|x|, \e (y)= \e(x))\}=\{y_{j} \}$ is finite. This follows because $\{y:|y| \leqslant |x|\}$ can be covered with a finite amount of neighborhoods of adapted coordinates, and in any of them the preimage of any point is a finite set.
At least one $y_{j}$ realizes the minimum distance from $p$ to $q= \e (x)$, and must be either $A_3(I)$ or NC (in other words, $y \in \mathcal{I}$).
We will show that there is an FCLC joining $x$ and one $y_{j} \in \mathcal{I}$, though it may not be the one with minimal radius.

\begin{dfn}
  We define some important sets:
  \[ S_{R} =B_{R} \cap \e^{-1} ( \e ( \mathcal{A}_{2} \cap B_{R} )) \]
  \[ V_{1}^{0} =\{x \in V_{1} : \e^{-1} ( \e (x)) \cap B_{|x|} \subset \NC
     \cup \mathcal{A}_{2} \} \]
  \[ \SAtwo =\{x \in \mathcal{A}_{2} :    \exists y \in
     \mathcal{A}_{2} ,   \e (y)= \e (x), |y|<|x|\} \]
  
  In other words, $V_{1}^{0}$ consists of those points $x \in V_{1}$ such that
  all preimages of $\e (x$) with radius smaller than $|x|$ are $\NC$ {\emph{or}}
  $\mathcal{A}_{2}$.
\end{dfn}

\begin{dfn}
  A GACDC is an ACDC $\alpha$ such that
\begin{itemizedot}
  \item $\tmop{Im} ( \alpha )$ is contained in $\mathcal{C} \cap V_{1}^{0}$.
  \item for any $y \in B_{| \alpha (t_{0} )|} \cap \mathcal{A}_{2}$ such that
     $\e ( \alpha (t_{0} ))= \e (y)$, $\e \circ \alpha$ is transverse to $\e
     (\mathcal{A}_{2} \cap B_{\varepsilon} (y))$ at $t_{0}$, for some
     $\varepsilon >0$.
\end{itemizedot}
\end{dfn}

In words, all possible retorts of a GACDC avoid all singularities that are not $A_{2}$ and only meet $\mathcal{A}_{2}$ transversally.

\begin{dfn}\label{def: linking curve algorithm}
The \emph{linking curve algorithm} is a procedure that attempts to build an FCLC starting at a given point $x\in V_1$ (see figure \ref{figure: algorithm for linking curves}).

It starts with the trivial \emph{aspirant curve} $\alpha =\{x\}$ and updates it at each segment by addition of one or more segments, to get a new aspirant curve.
It only stops if the aspirant curve is saturated, and its tip is in $\mathcal{I}$.

The aspirant curve $\alpha = \alpha_{1} \ast \ldots \ast \alpha_{k}$ is updated following the only rule in the following list that can be applied:
\begin{description}
  \item[Descent] If the tip of $\alpha_{k}$ is a point in $\mathcal{J}$, let $\gamma$ be a GACDC contained in $V_{1}^{0}$ that starts at $x$ and ends up in an $A_3$ point.
  We know that $\gamma$ intersects $\SAtwo$ in a finite set and, for convenience, we split $\gamma$ into $r$ GACDCs $\alpha_{k+1},\dots,\alpha_{k+r}$ such that each of these curves intersects $\SAtwo$ only at its extrema.
  The new curve $\alpha \ast \alpha_{k+1} \ast \dots \ast \alpha_{k+r}$ ends up in an $A_{3}$ point.
  The next step is an $A_3$ join.

  \item[$A_3$ join] %
  If $\alpha_{k} :[0,T] \rightarrow V_{1}$ is a ACDC ending up in an $A_{3}$ point, add the retort $\alpha_{k+1}$ of $\alpha_{k}$ that starts at the $A_{3}$ join $\alpha_k(T)$.
  This is always possible, since $\alpha_{k}$ does not intersect $\SAtwo$.
  The new tip of $\alpha \ast \alpha_{k+1}$ will be $NC$, $A_2$ or $A_3$, but the latter can only happen if $\alpha \ast \alpha_{k+1}$ is a linking curve.

  \item[Reprise] If the tip of $\alpha$ is $NC$ and $\alpha$ is not a linking curve, let $\alpha_j$ be the latest loose curve in $\alpha$.
  We add the retort $\alpha_{k+1}$ of $\alpha_{j}$ starting at the tip of $\alpha$.
  This is possible for the same reason as before and, again, the new tip of $\alpha \ast \alpha_{k+1}$ will be $NC$, $A_2$ or $A_3$, and the latter can not happen unless $\alpha \ast \alpha_{k+1}$ is a linking curve.

  \item[Success!] If $\alpha$ is saturated and its tip is in $\mathcal{I}$, then $\alpha$ is an FCLC, so we report success and stop the algorithm.
  For completeness, the algorithm also reports success if $\alpha =\{x\}$, for
  $x \in \mathcal{I}$.
\end{description} 
\end{dfn}

\begin{figure}[ht]
 \includegraphics[width=0.9\textwidth]{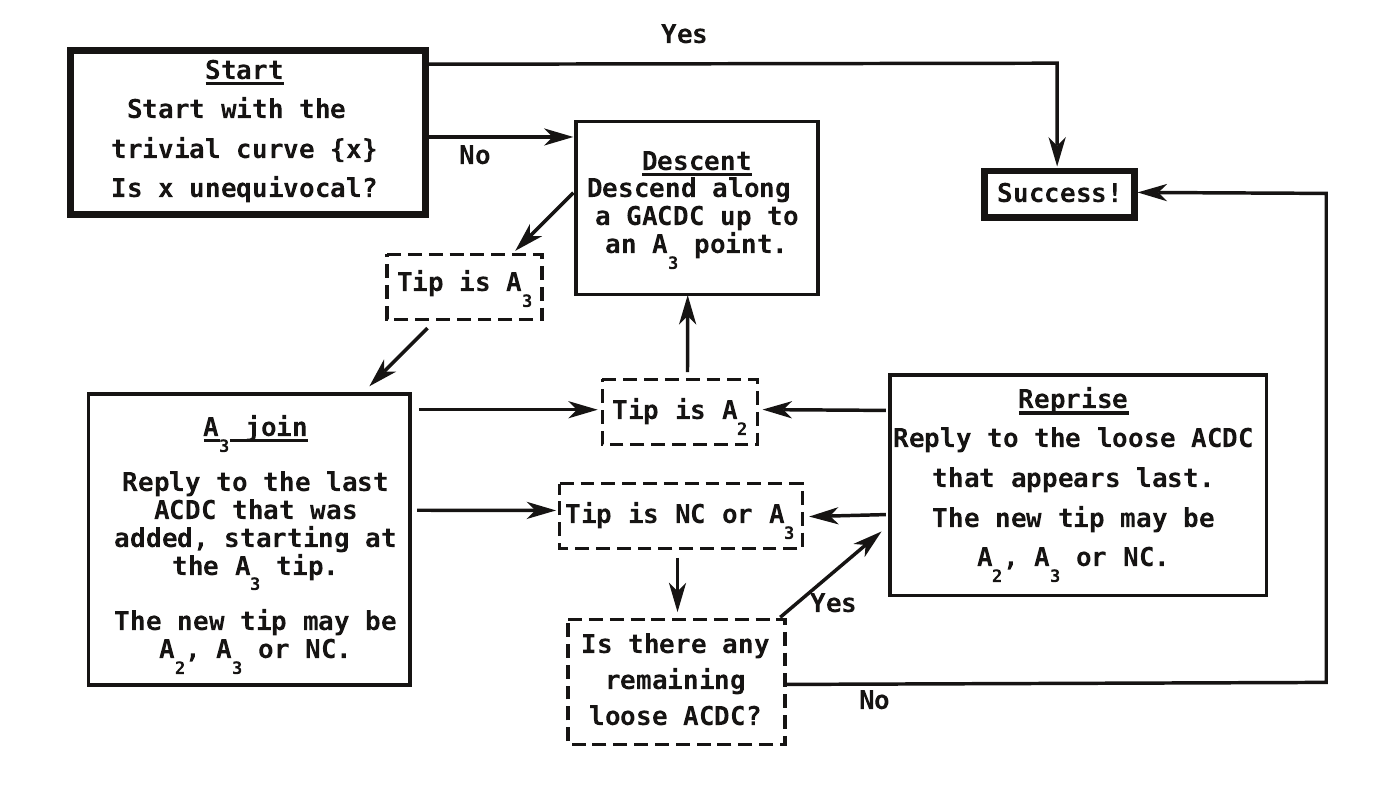}
 \caption{Flow diagram for the linking curve algorithm}
 \label{figure: algorithm for linking curves}
\end{figure}

\begin{remark}
The algorithm can also be presented in a \emph{recursive} fashion. We start with some definitions:

\begin{itemize}
 \item $Tip(\alpha)=\alpha(T)$, for any curve $\alpha$ defined in an interval $[0,T]$.
 \item $Ret(\alpha, y)$ is the retort of $\alpha$ starting at $y$, for any curve $\alpha$ contained in $V^1_0\setminus\SAtwo$, and a point $y\in V_1$ such that $\e(y)=\e(Tip(\alpha))$.
\end{itemize}

Then for any $x\in V_1$, we define an aspirant curve $L(x)$ by the following rules:
\begin{itemize}
 \item If $x\in\mathcal{J}$, then $L(x)=\{x\}$
 \item If $x\in \mathcal{I}$, then compute the GACDC curve $\gamma=\gamma_{1} \ast \dots \ast \gamma_{r}$, as above. Then $L(x)=\gamma_1\ast L(Tip(\gamma_1)) \ast Ret(\gamma_1,Tip(L(Tip(\gamma_1))))$
\end{itemize}

 The reader may have noticed that $\gamma_{2}$ to $\gamma_{r}$ are discarded, and only $\gamma_1$ is kept (the ACDC up to the first point in $\SAtwo$).
 This actually causes a small technical problem, so we will use only the iterative version of the algorithm.
\end{remark}

\begin{theorem}\label{thm: the linking curve algorithm stops in finite time}
  Let $M$ be a manifold with a Riemannian metric in $\mathcal{H}_{M}$.
\begin{enumerate}
 \item   For any $R>0$ there is $L>0$ such that any GACDC starting at $x \in \mathcal{J} \cap B_{R}$ has length at most $L$, and can be extended until it reaches an $A_{3}$ point.
 \item The algorithm \ref{def: linking curve algorithm} always reports ``success!'' after a finite number of steps, for any starting point $x \in \mathcal{J}$.
\end{enumerate}
\end{theorem}

\begin{dfn}
A pair ($S,O$) of open subsets of $\T$ with $\bar{S} \subset O$, is {\strong{transient}} if and only if for any point $x$ in $S \cap \mathcal{J}$, any aspirant curve that starts at $x$ can be extended to either an aspirant curve with endpoint outside of $O$ or an FCLC contained in $O$.

  The {\strong{gain}} of a transient pair $(S,O)$ is the infimum of all
  $|x|-|y|$, for all $x \in S$, $y \in V_{1} \setminus O$ such that there is
  an aspirant curve starting at $x$ and ending at $y$.

  A transient pair is {\strong{positive}} if it has positive gain.

  A transient pair is \strong{bounded} if there is a uniform bound for the length of any aspirant curve contained in $O$.
\end{dfn}

\begin{lem}
  \label{typical sings have transient neighborhoods}
  For any point $x$ of type NC, $A_{2}$, $A_{3}$, $A_{4}$, $D_{4}^{+}$ or $D_{4}^{-}$,
  there is a positive bounded transient pair $(S,O)$, with $x \in S$.
\end{lem}

Lemma \ref{typical sings have transient neighborhoods} is all we need to complete the proof of Main Theorem \ref{main theorem generic}:

\begin{proof}[Proof of theorem \ref{thm: the linking curve algorithm stops in finite time}]
We prove the second part first.

Define:
\[ R_{0} = \sup \left\{ R: \text{ for all }x \in B_{R} ,  \begin{array}{l}
     \text{the algorithm starting at $x$ reports }\\
     \text{success! after a finite amount of iterations}
   \end{array} \right\} \]

We will assume that $R_{0}$ is finite and derive a contradiction, thus showing that the algorithm always reports success after a finite amount of iterations.
Using lemma \ref{typical sings have transient neighborhoods}, we cover $\overline{B_{R_{0}}}$ by a finite number of neighborhoods $\{S_{i} \}_{i=1}^{N}$, where $(S_{i} ,O_{i} )$ are bounded positive transient pairs.
Then $B_{R_{0} + \varepsilon}$ is also covered by $\cup S_{i}$ for some $\varepsilon >0$.
Let $\varepsilon_{0}$ be the minimum of $\varepsilon$, and all the gains of the $N$ pairs $(S_{i} ,O_{i} )$.

Take a point $x \in B_{R_{0} + \varepsilon_{0}}$ and assume $x \in S_{1}$.
By hypothesis we can find an aspirant curve $\alpha$ with endpoint $y$ outside of $O_{1}$.

Thanks to the way we have chosen $\varepsilon_0$, we can assume $|y|<R_{0}$, and by hypothesis there is a saturated aspirant curve $\beta$ that joins $y$ to some point $z$.
Then $\alpha\ast\beta$ is an aspirant curve starting at $x$ and ending at $z$.
If we want to complete this aspirant curve to get a saturated one, it remains to reply to all the loose ACDCs in $\alpha$.
Each of them, except possibly its endpoints, is contained in $V^1_0\setminus\SAtwo$.
If, after replying to one of them, we hit an $A_2$ point $y_{0}$, then $|y_{0}|<R_{0}$, and thus we can append a saturated aspirant curve that joins $y_{0}$ to some $z_{0} \in \mathcal{N C} \cap B_{|y_{0} |}$.
Then we can continue to reply to the remaining loose ACDCs, and the process finishes in a finite number of steps.
This is the desired contradiction that completes the proof of the second part.

The first part follows trivially because the covering is by bounded pairs.%

\end{proof}

\begin{proof}[Proof of theorem \ref{claim: hypothesis of the synthesis theorem}]

The first part of theorem \ref{thm: the linking curve algorithm stops in finite time} guarantees that we can always perform the ``{\emph{descent}}'' step in the diagram.
We have already shown why the other steps can always be performed.

The second part of that theorem shows that the algorithm always stops after a finite number of iterations.

Thus, we can always produce an FCLC starting at any point in $\mathcal{J}$.
Theorem \ref{main properties of linking curves} shows that an FCLC is a linking curve.

This, together with lemma \ref{A3I are unequivocal} completes the proof of theorem \ref{claim: hypothesis of the synthesis theorem}.
\end{proof}

It only remains to prove lemma \ref{typical sings have transient neighborhoods}.
Before we can prove it, we need to look at $A_4$ and $D_4$ points more closely.

\subsection{CDCs in adapted coordinates\label{section: CDCs in adapted coords} for $A_4$ and $D_4$ points.}

As we mentioned in section \ref{section: generic metric}, the radial vector field, and the spheres of constant radius of $\T$, which have very simple expressions in standard linear coordinates in $\T$, are distorted in adapted coordinates.
Thus, the distribution $D$ and the CDCs do not always have the same expression in adapted coordinates.
In this section, we study them qualitatively.
We will use the name $R: \T \rightarrow \mathbb{R}$ for the radius function, and $r$ for the radial vector field, and we assume that our conjugate point is a first conjugate point (it lies in $\partial V_{1}$).

\subsubsection{$A_{4}$ points}

In a neighborhood $O$ of an $A_{4}$ point, $\Tone$ can be stratified as an isolated $A_{4}$ point, inside a stratum of dimension $1$ of $A_{3}$ points, inside a smooth surface consisting otherwise on $A_{2}$ points.
The conjugate points are given by $4x_{1}^{3} +2x_{1} x_{2} +x_{3} =0$, and the $A_{3}$ points are given by the additional equation $12x_{1}^{2} +2x_{2} =0$.
The kernel is generated by the vector $\frac{\partial}{\partial x_{1}}$ at any conjugate point and we can assume that $D$ is close to $\frac{\partial}{\partial x_{1}}$ in $O \cap\mathcal{C}$.

The radial vector field does not have a fixed expression in adapted coordinates, but the distribution $D$ is a smooth line distribution and its integral curves are smooth.
Thus, the $A_{4}$ point belongs to exactly one integral curve of $D$.

\begin{figure}[ht]
 \centering
 \includegraphics[width=0.8\textwidth]{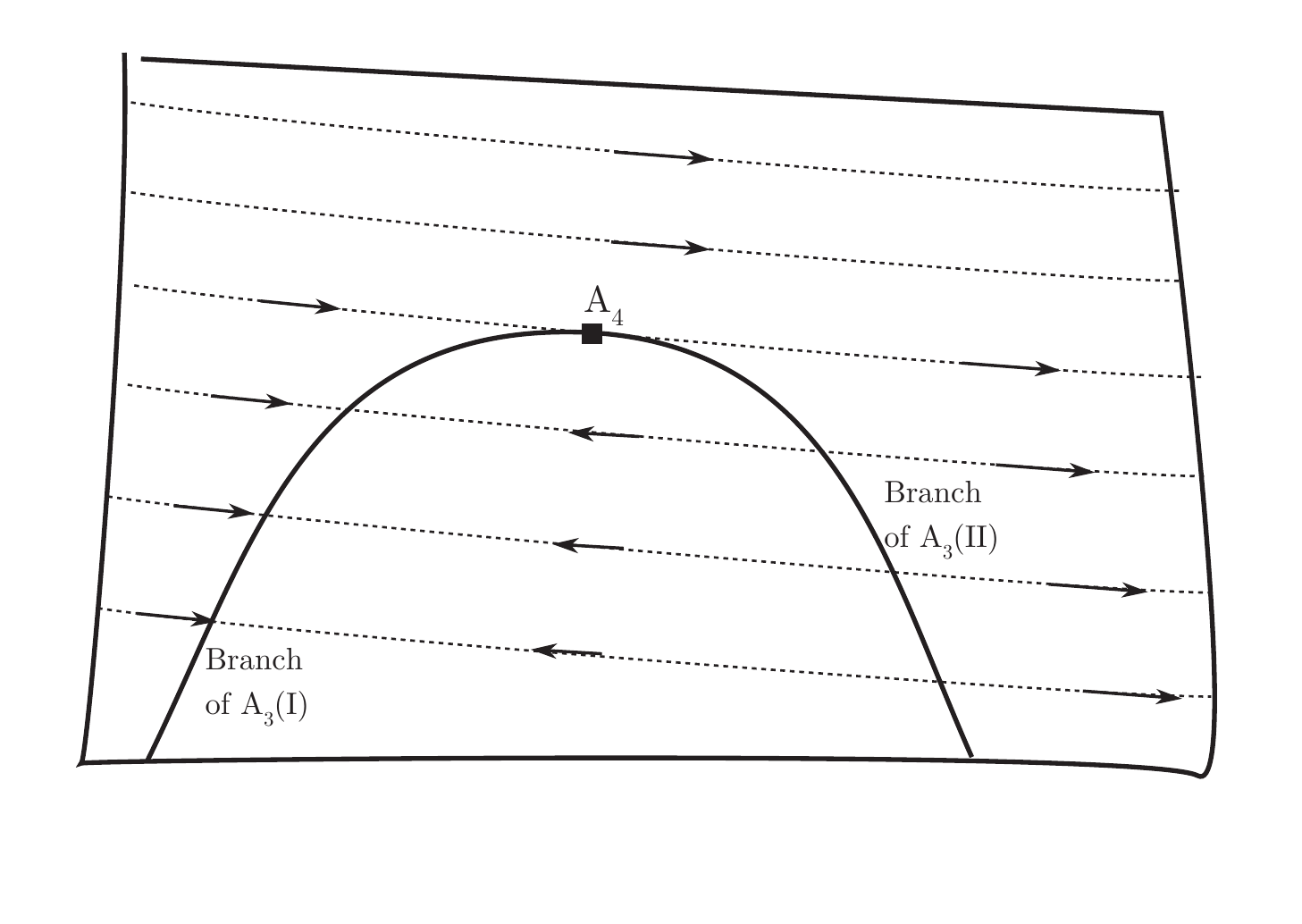}
 \caption{The
distribution $D$ and the CDCs at the conjugate points near an $A_{4}$
point.}
 \label{figure: near an A4 point}
\end{figure}

As we saw, $A_3(I)$ (resp $A_3(II)$) points have neighborhoods without $A_3(II)$
(resp $A_3(I)$) points. The $A_{4}$ point splits $\mathcal{A}_{3}$ into two
branches, and it can be shown easily that they must be of different types.
Composing with the coordinate change ($x_{1} ,x_{2} ,x_{3} ) \rightarrow
(-x_{1} ,x_{2} ,x_{3}$) if necessary, we can assume that the CDCs travel in
the directions shown in figure \ref{figure: near an A4 point}.

\subsubsection{$D_{4}^{-}$ points}

In a neighborhood $O$ of adapted coordinates near a $D_{4}^{-}$ point,
$\mathcal{C}$ is a cone given by the equations $0=-x_{1}^{2} -x_{2}^{2}
+x_{3}^{2}$. The kernel of $d  e_{1}$ at the origin is the plane $x_{3} =0$,
which intersects this cone only at ($0, \ldots ,0$). Three generatrices of the
cone consist of $A_{3}$ points (they are given by the equations $x_{2} =0$,
$x_{1} -x_{3} =0$ and $2x_{1} +x_{3}=0$, plus the equation of the cone), and the
rest of the points are $A_{2}$.

The radial vector field ($r_{1} ,r_{2} ,r_{3}$) at the origin must lie within
the solid cone $-r_{1}^{2} -r_{2}^{2} +r_{3}^{2} >0$, because the number of
conjugate points (counting multiplicities) in a radial line through a point
close to ($0,0,0$), must be $2$. In particular, $|r_{3} |>0$. Composing with
the coordinate change ($x_{1} ,x_{2} ,x_{3} ) \rightarrow (-x_{1} ,-x_{2}
,-x_{3}$) to the left and ($x_{1} ,x_{2} ,x_{3} ) \rightarrow (x_{1} ,x_{2}
,-x_{3}$) to the right, if necessary, we can assume that $r_{3} >0$.

The kernel at the origin is contained in the tangent to the hypersurface
$T_{0} = \{ R ( y ) =R ( 0 ) \}$, and the radius always decreases along a CDC.
Thus a CDC starting at a first conjugate point moves away from the origin and
may either hit an $A_{3}$ point, or leave the neighborhood. Thus these points
are not sinks of CDCs starting at points in $V_{1}$.

We now claim that there are three CDCs that start at any $D_{4}^{-}$ point and
flow out of $O$, and three CDCs that flow into any $D_{4}^{-}$ point, but the
latter ones are contained in the set of second conjugate points.

Recall that the $D_{4}^{-}$ point is the origin. We write the radial vector as
its value at the origin plus a first order perturbation:
\[ r=r^{0} +P(x) \]
with $|P(x)|<C|x|$ for some constant $C$.

We will consider angles and norms in $O$ measured in the adapted coordinates
in order to derive some qualitative behavior, even though these quantities do
not have any intrinsic meaning.

We can measure the angle between a generatrix $G$ and $D$ by the determinant
of a vector in the direction of $G$, the radial vector $r$ and the kernel $k$
of $ e_{1}$: the determinant is zero if and only if the angle is zero. The
angle between $k$ and $r$ in this coordinate system is bounded from below, and
the norm of $r$ is bounded close to $1$. Thus if we use unit vectors that span
$G$ and $k$, we get a number $d(x)$ that is comparable to the sine of the
angle between $G$ and the plane spanned by $r$ and $k$. Thus $c|d(x)|$ is a
bound from below to $| \sin ( \alpha )|$, where $\alpha$ is the angle between
$G$ and $D$, for some $c>0$.

The kernel is spanned by $(-x_{1} +x_{3} ,x_{2} ,0)$ if $-x_{1} +x_{3} \neq
0$. The generatrix of $C$ at a point $(x_{1} ,x_{2} ,x_{3} ) \in C$ is the
line through $(x_{1} ,x_{2} ,x_{3} )$ and the origin. So $d$ is computed as
follows:
\[ d(x)= \frac{1}{x_{1}^{2} +x_{2}^{2} +x_{3}^{2}} \left|\begin{array}{ccc}
     x_{1} & x_{2} & x_{3}\\
     -x_{1} +x_{3} & x_{2} & 0\\
     r_{1} & r_{2} & r_{3}
   \end{array}\right| \]

Let us look for the roots of the lower order ($0$-th order) approximation:
\[ d_{0} (x)= \frac{1}{x_{1}^{2} +x_{2}^{2} +x_{3}^{2}}
   \left|\begin{array}{ccc}
     x_{1} & x_{2} & x_{3}\\
     -x_{1} +x_{3} & x_{2} & 0\\
     r^{0}_{1} & r^{0}_{2} & r^{0}_{3}
   \end{array}\right| \]

where $(r^{0}_{1} ,r^{0}_{3} ,r^{0}_{3} )$ are the coordinates of $r^{0}$.

The equation $\left|\begin{array}{ccc}
  x_{1} & x_{2} & x_{3}\\
  -x_{1} +x_{3} & x_{2} & 0\\
  r^{0}_{1} & r^{0}_{2} & r^{0}_{3}
\end{array}\right| =0$ is homogeneous in the variables $x_{1}$, $x_{2}$ and
$x_{3}$, so we can make the substitution $-x_{1} +x_{3} =1$ in order to study
its solutions. We only miss the direction $\lambda (1,0,1)$, where $D$ is not
aligned with $G$ because it consists of $A_{3}$ points.

Points in $C$ now satisfy $1+2x_{1} -x_{2}^{2} =0$, and $d_{0} (x)$ becomes
$p(x_{2} )=- \frac{1}{2}   \hspace{0.25em} ( a - 1 )  x_{2}^{3}  + 
\frac{1}{2}   \hspace{0.25em} b x_{2}^{2}  -  \frac{1}{2}   \hspace{0.25em} (
a + 3 )  x_{2}  +  \frac{1}{2}   \hspace{0.25em} b$, for $a=
\frac{r^{0}_{1}}{r^{0}_{3}}$ and $b= \frac{r^{0}_{2}}{r^{0}_{3}}$ (recall
$r^{0}_{3} >0$). The lines of $A_{3}$ points correspond to $x_{2} =
\frac{-1}{\sqrt{3}}$, $x_{2} = \frac{1}{\sqrt{3}}$, and the third line lies at
$\infty$. We prove that $p$ has three different roots, one in each interval:
($- \infty , \frac{-1}{\sqrt{3}}$), ($\frac{-1}{\sqrt{3}} ,
\frac{1}{\sqrt{3}}$), ($\frac{1}{\sqrt{3}} , \infty$). This follows
immediately if we prove $\lim_{x_{2} \rightarrow - \infty} p(x_{2} )=-
\infty$, $p( \frac{-1}{\sqrt{3}} )>0$, $p( \frac{1}{\sqrt{3}} )<0$ and
$\lim_{x_{2} \rightarrow \infty} p(x_{2} )= \infty$ for all $a$ and $b$ such
that $a^{2} +b^{2} <1$. The first and last one are obvious, so let us look at
the second one. The minimum of
\[ p( \frac{-1}{\sqrt{3}} )= \frac{2 \sqrt{3}}{9} a+ \frac{2}{3} b+ \frac{4
   \sqrt{3}}{9} \]
in the circle $a^{2} +b^{2} \leqslant 1$ can be found using Lagrange
multipliers: it is exactly $0$ and is attained only at the boundary $a^{2}
+b^{2} =1$. The third inequality is analogous.

Thus, there is exactly one direction where $D$ is aligned with $G$ en each
sector between two lines of $A_{3}$ points. Take polar coordinates ($\phi ,r$)
in $C \cap V_{1}$. The roots of $d_{0}$ are transverse, and thus if
$\phi_{0}$ corresponds to a root of $_{} d_{0}$, then at a line in direction
$\phi$ close to $\phi_{0}$, the angle between $D$ and $G$ is at least $c( \phi
- \phi_{0} )+ \eta ( \phi ,r$), for $c>0$ and $\eta ( \phi ,r)=o(r$). If, at a
point in the line with angle $\phi$, and sufficiently small $r>0$, we move
upwards in the direction of $D$ (in the direction of increasing radius), we
hit the line of $A_{3}$ points, not the center. There are two CDCs starting at
each side of every $A_{3}$ point. A continuity argument shows that there must
be one CDC in each sector that starts at the origin (see figure \ref{figure:
near elliptic umbilic}).

\begin{figure} %
\includegraphics[width=0.8\textwidth]{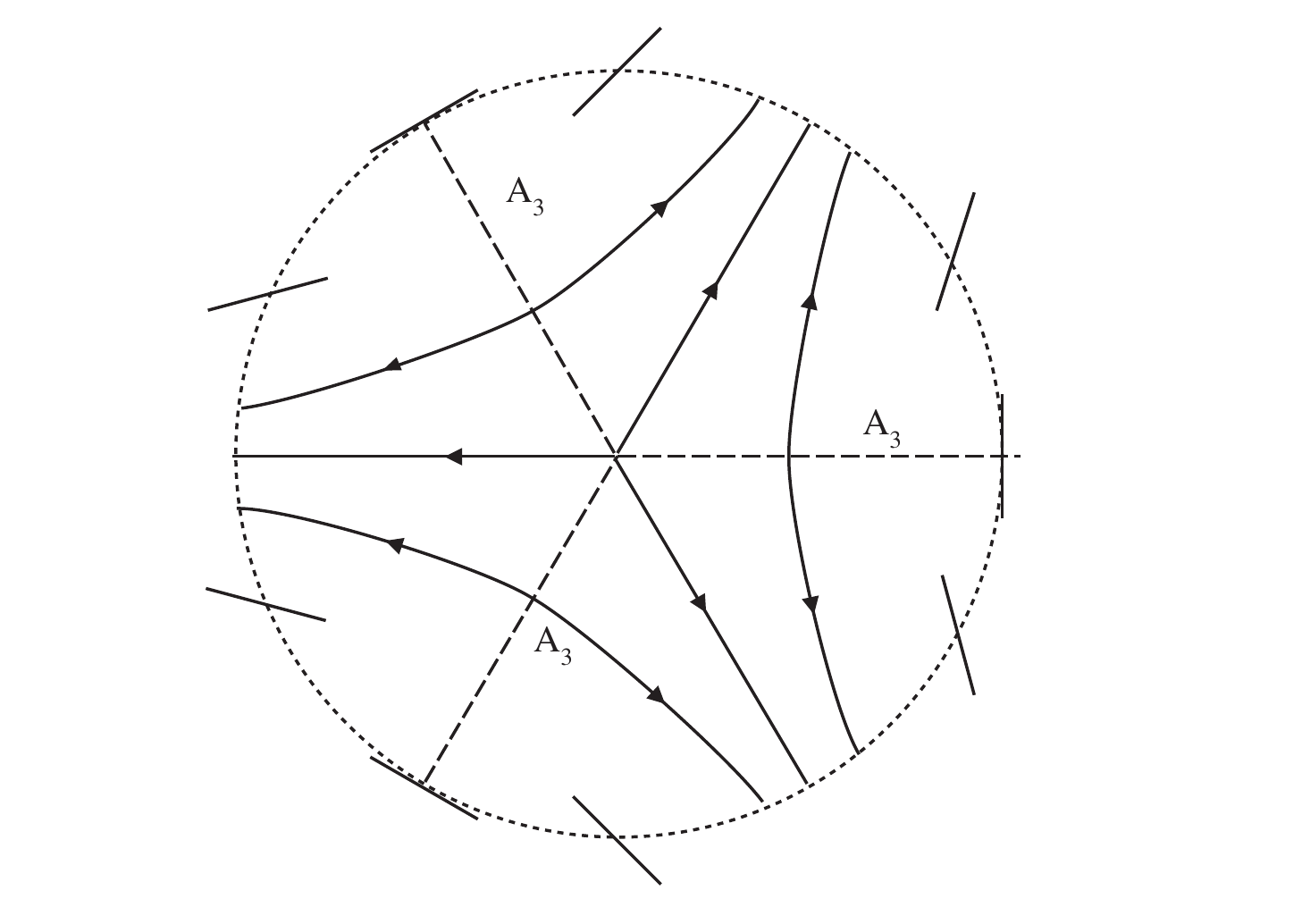}
 \caption{CDCs in the half-cone of first conjugate points near an
elliptic umbilic point, using the chart
$(x_{1} ,x_{2} ) \rightarrow (x_{1} ,x_{2} ,- \sqrt{x_{1}^{2} +x_{2}^{2}} )$,
for $r_{0} =(0,0,1)$.
The distribution $D$ makes half turn as we make a full turn around $x_{1}^{2}
+x_{2}^{2} =1$, spinning in
the opposite direction.}
 \label{figure: near elliptic umbilic}
\end{figure}

Reversing the argument, we see that there are three CDCs that descend into the
elliptic umbilic point, one in each sector, all contained in the the set of
second conjugate points.

\subsubsection{$D_{4}^{+}$ points}

The conjugate points in a neighborhood of adapted coordinates lie in the
cone $C$ given by $0=x_{1} x_{2} -x_{3}^{2} = \frac{1}{4} (x_{1} +x_{2} )^{2}
- \frac{1}{4} (x_{1} -x_{2} )^{2} -x_{3}^{2}$. This time, the kernel of $d \e$
at the origin intersects this cone in two lines through the origin, and the
inside of the cone $x_{1} x_{2} -x_{3}^{2} >0$ is split into two parts. There
is one line of $A_{3}$ points, the generatrix of the cone with parametric
equations: $t \rightarrow (t,t,t)$.

The radial vector at $r=(r_{1} ,r_{2} ,r_{3} )$ must lie within the solid cone
$r_{1} r_{2} -r_{3}^{2} >0$, for the same reason as above. Composing with the
coordinate change $(x_{1} ,x_{2} ,x_{3} ) \rightarrow (-x_{1} ,-x_{2} ,-x_{3}
)$ to the left and $(x_{1} ,x_{2} ,x_{3} ) \rightarrow (x_{1} ,x_{2} ,-x_{3}
)$ to the right, if necessary, we can assume that $r_{1} >0$ and $r_{2} >0$.

We write the radial vector as its value at the origin plus a first order
perturbation:
\[ r=r^{0} +P(x) \]
with $|P(x)|<C|x|$ for some constant $C$.

As before, the radius decreases along a CDC, but this time, a CDC starting at
a first conjugate point might end up at the origin. Let $F$ be the half cone
of first conjugate points (given by the equations $x_{1} x_{2} =x_{3}^2$ and
$\frac{1}{2} (x_{1} +x_{2} ) <0$). Let $F_{+}$ be the points of $F$ with
radius greater than the origin. Its tangent cone at the origin is $F \cap
\{x_{3} <0\}$ or $F \cap \{x_{3} >0\}$, depending on the sign of the third
coordinate of $r_{0}$.

As in the previous case, we can measure the angle between a generatrix $G$ and
$D$ by the determinant of a vector in the direction of $G$, the radial vector
$r$ and the kernel $k$ of $ e_{1}$. This time, the kernel is spanned by
$(-x_{3} ,x_{1} ,0)$ in the chart $x_{1} \neq 0$.
$$ d(x)= \frac{1}{x_{1}^{2} +x_{2}^{2} +x_{3}^{2}} \left|\begin{array}{ccc}
     x_{1} & x_{2} & x_{3}\\
     -x_{3} & x_{1} & 0\\
     r_{1} & r_{2} & r_{3}
   \end{array}\right|
$$

Again, we look for the roots of the lower order ($0$-th order) approximation,
which is equivalent to looking for the zeros of:
\[ \tilde{d} (x)= \left|\begin{array}{ccc}
     x_{1} & x_{2} & x_{3}\\
     -x_{3} & x_{1} & 0\\
     a & b & 1
   \end{array}\right| \]
in the cone $C$, for $a= \frac{r^{0}_{1}}{r^{0}_{3}}   \tmop{and}   b=
\frac{r^{0}_{2}}{r^{0}_{3}}$. We can make the substitution $x_{1} =-1$ in
order to study the zeros of the polynomial (we choose $x_{1} <0$ because we
are interested in the half cone of first conjugate points). This implies
$x_{2} =-x_{3}^{2}$ for a point in $C$, and we are left with $p(x_{3}
)=-x_{3}^{3} -bx_{3}^{2} +ax_{3} +1=0$. If $b^{2} +3a>0$, $p$ has two critical
points $\frac{-b \pm \sqrt{b^{2} +3a}}{3}$, otherwise it is monotone
decreasing. But even when $p$ has two critical points, the local maximum may
be negative, or the local minimum positive, with one real root.

The vector $r^{0}$ must satisfy $r_{3}^{0} \neq 0$ and $x_{1} x_{2} -x_{3}^{2}
>0$, or $ab>1$. There are two {\emph{chambers}} for $r^{0}$: $r_{3}^{0} >0$
and $r_{3}^{0} <0$. We will say that a $D_{4}^{+}$ point such that $r_{3}^{0}
>0$ (resp, $r_{3}^{0} <0$) is of type I (resp, type II).

If $r_{3}^{0} >0$ (or $a,b>0$), then $r^{0}$ and $L \cap F$ lie at opposite
sides of the kernel of $d  e_{1}$ at the origin. The cubic polynomial $p$ has
limit $\mp \infty$ at $\pm \infty$, \ and $p(0)>0$. The line of $A_{3}$ points
intersects $x_{1} =-1$ at $x_{3} =-1$. We check that $p(x_{3} =-1)=2-a-b$ is
always negative in the region $a>0$, $b>0$, $ab>1$. Thus there is exactly one
positive root, and two negative ones, one at each side of the line of $A_{3}$
points. This corresponds to the top right picture in figure \ref{figure:
hyperbolic umbilic}, where the $x_{3}$ axis is vertical, and the CDCs descend,
because $r_{3}^{0} >0$.

\begin{figure}%
\resizebox{\textwidth}{!}{
  \begin{tabular}{ll}
    \includegraphics{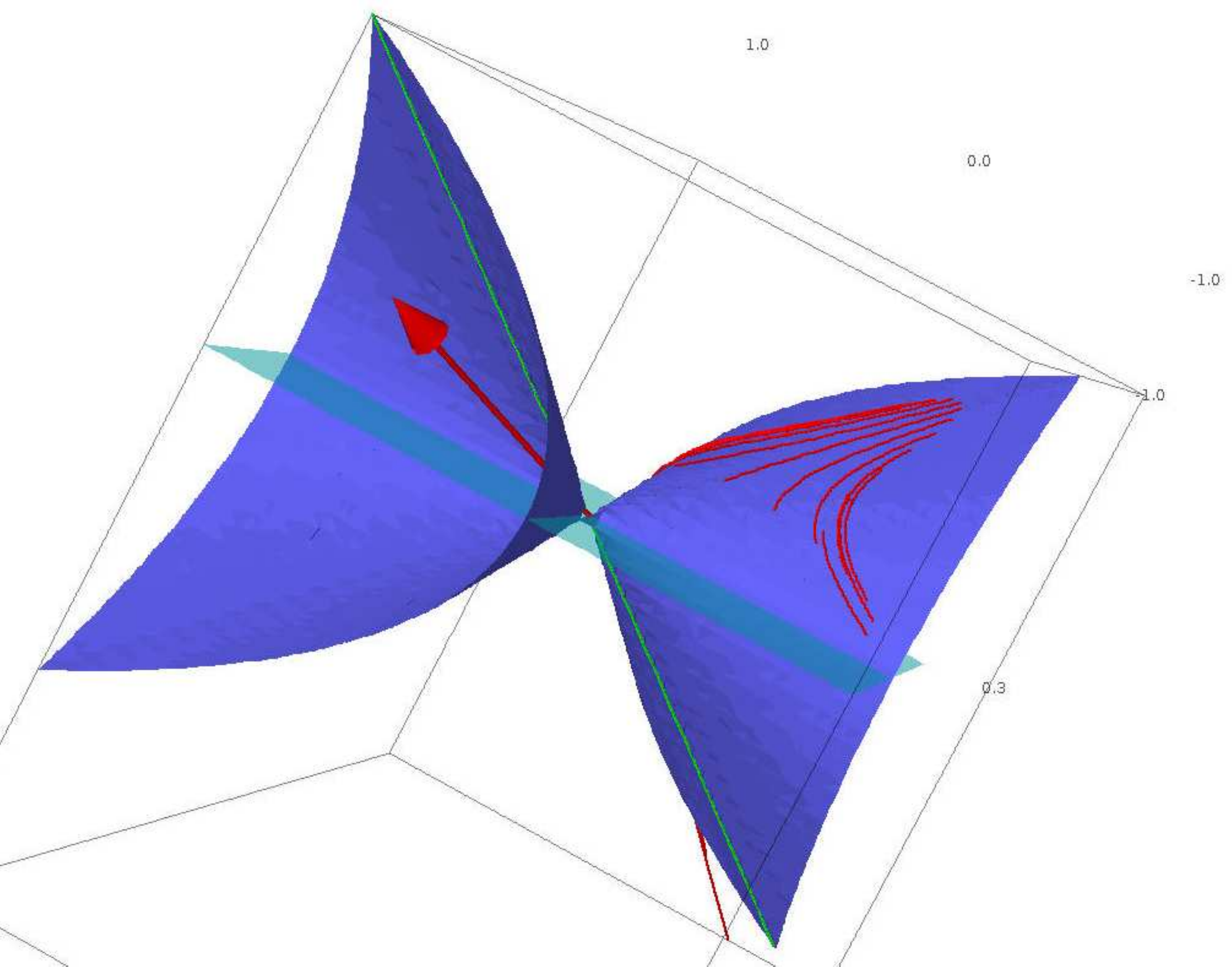} &
    \includegraphics{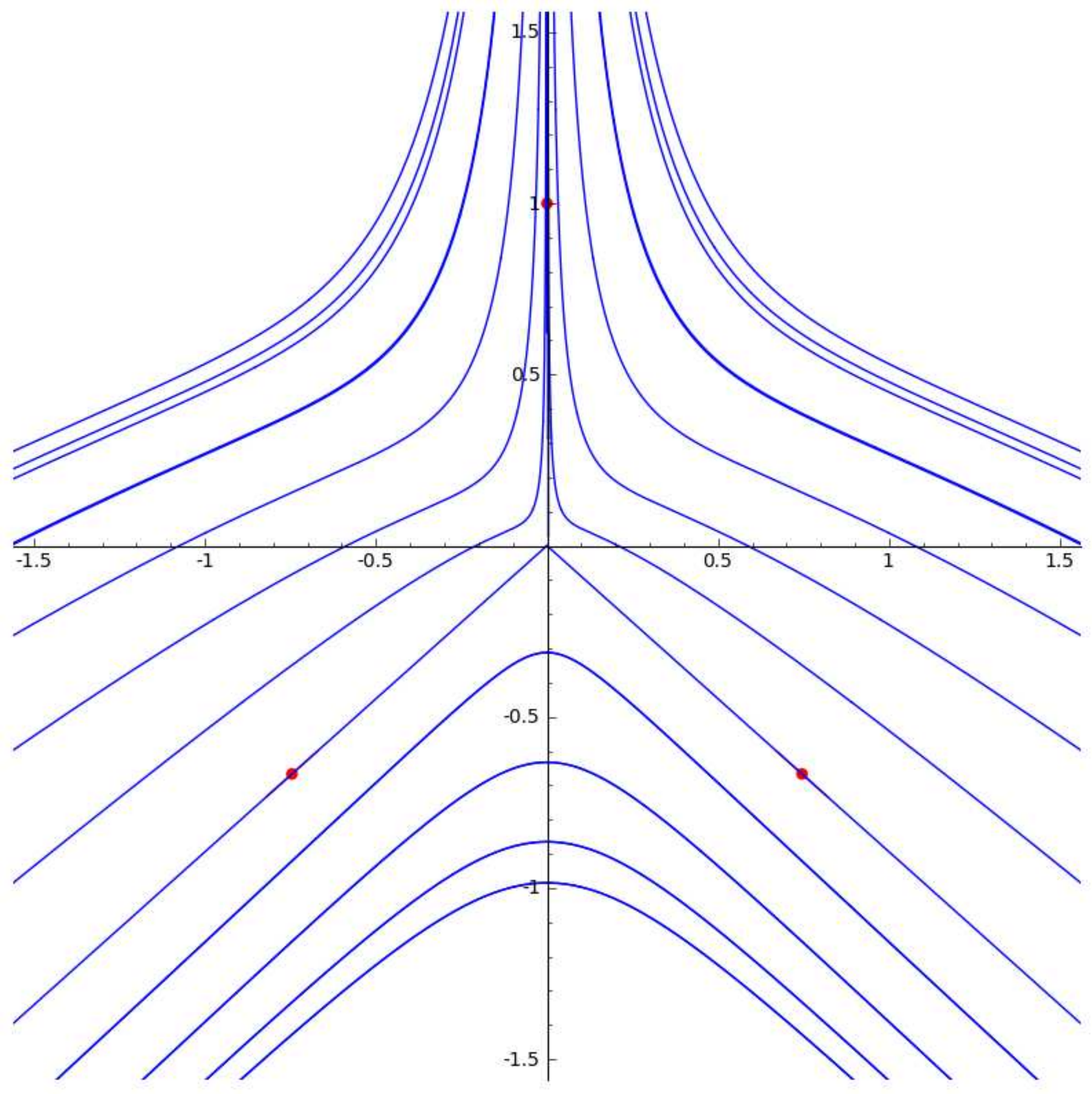}\\
    & \\
    \includegraphics{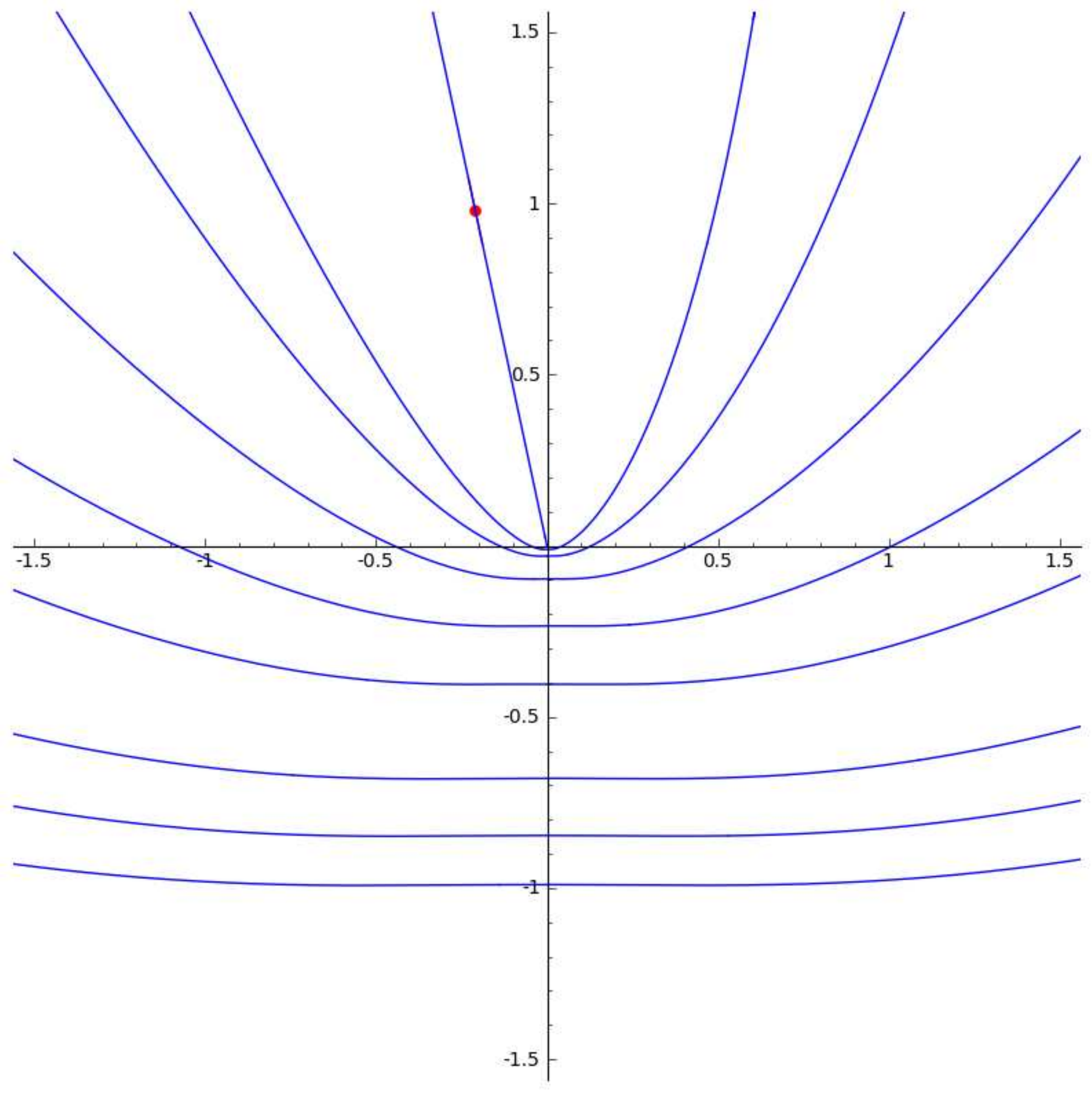} &
    \includegraphics{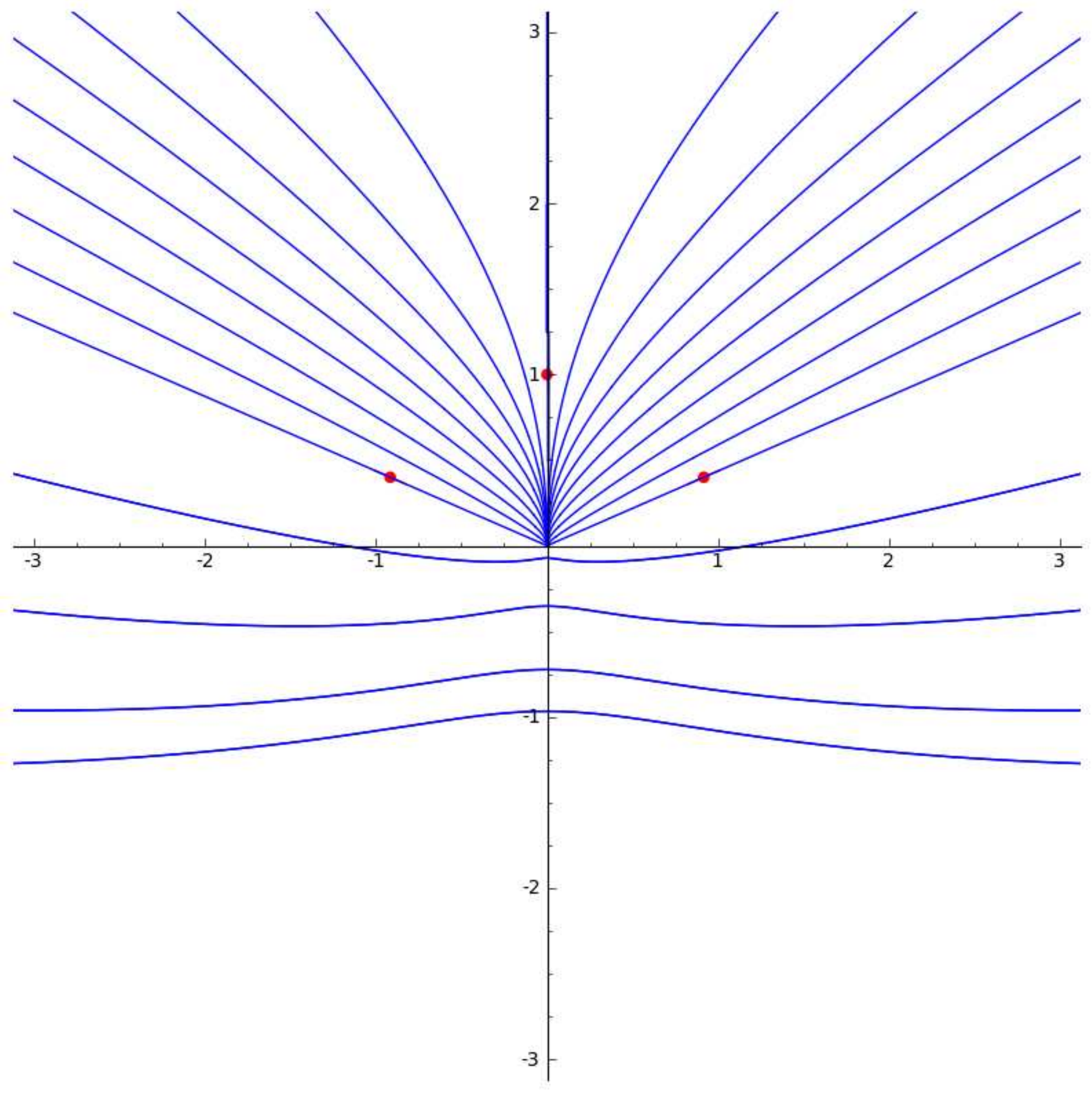}
  \end{tabular}}
  \caption{An hyperbolic umbilic point.}
  \label{figure: hyperbolic umbilic}
\begin{minipage}[b]{0.9\linewidth}  %
\mbox{}\\[-.5\baselineskip] %
\paragraph{Explanation of figure \ref{figure: hyperbolic umbilic}}
 In the \strong{TopLeft} corner, the cone $C$ appears in blue, the line of
$A_{3}$ points in green,
  the radial vector at the origin in red, and the CDCs in red.

  The other pictures show the CDCs in the parametrization of the half cone of
  first conjugate points, obtained by projecting onto the plane spanned by
  $(1,-1,0)$ and $(0,0,1)$. The red dots indicate the directions where $D$ is
  parallel to the generatrix of the cone. The $A_{3}$ points lie in the half
  vertical line with $x_{3} <0$.
  \begin{description}
   \item[TopRight] $a>0$, $b>0$.
   \item[BottomLeft]$a<0$, $b<0$, $p$ has only one real root.
   \item[BottomRight]$a<0$, $b<0$, $p$ has three distinct real roots.
  \end{description}
\end{minipage}  %
\end{figure}  %

The positive root gives a direction that is tangent to a CDC that enters into
the $D^{+}_{4}$ point, but moving to a nearby point we find CDCs that miss the
origin, and approach either of the two CDCs that depart from the origin,
corresponding to the negative roots of $p$.

However, if $r_{3}^{0} <0$ (type II), $p$ may have one or three roots.
We revert the direction of the CDC taking $p(z)=p(-x_3)$.
We note that $p(0)=1>0$, and $p' (z)>0$ for $z>0$, $a<0$ and $b<0$,
so there cannot be any positive root. A CDC starting at a point in $F$ flows
away from the stratum of $A_{3}$ points and out of the neighborhood (see the
bottom pictures at figure \ref{figure: hyperbolic umbilic}). It can be checked
by example that both possibilities do occur.

We want to remark that if there are three roots, the $D_{4}^{+}$ point is the
endpoint of the CDCs starting at any point in a set of positive
$\mathcal{\mathcal{H}}^{2}$ measure. Fortunately, all these points are second
conjugate points. This is the main reason why we build the synthesis as a
quotient of $V_{1}$ rather than all of $\T$ but more important: this is a hint
of the kind of complications we might find in arbitrary dimension, or for an
arbitrary metric, where we cannot list the normal forms and study each
possible singularity separately.

\begin{remark}
  In order to find out the number of real roots of $p$, for any value of $a$ and $b$, we used Sturm's method. 
  However, once we found out the results, we found alternative proofs and did not need to mention Sturm's method in the proof. 
  The precise boundary between the sets of $a,b$ such that $p$ has one or three real roots is found by Sturm's method. 
  It is given by:
  
  $p_{3} =-9  \hspace{0.25em} a^{2}  b^{2}  - 36  \hspace{0.25em} a^{3}  - 36 
  \hspace{0.25em} b^{3}  - 162  \hspace{0.25em} a b + 243=0$
\end{remark}

\subsection{Proof of lemma \ref{typical sings have transient neighborhoods}}

Let $x \in V_{1}$ be a point and $O$ be a cubical neighborhood of adapted coordinates around it.
$S$ will be a ``small enough'' subset of $O$:
\begin{description}
  \item[$NC$] The algorithm reports success! in one step for any non-conjugate point, so any $S
  \subset O$, such that $O$ has no conjugate points, satisfies the claim.
  The gain is the infimum of the empty set, $+\infty$, so the pair is positive.
  
  \item[$A_{2}$] The CDC $\alpha_{0}$ starting at $x_{0}$ that reaches
  $\partial O$ has a length $\varepsilon >0$. For $x$ in a sufficiently small
  neighborhood $S$ of $x_{0}$, there is a GACDC $\alpha$ that reaches $y \in
  \partial O$ and has length at least $\varepsilon /2$.
  
  If there is an aspirant curve that starts with $\alpha$, and later has a
  retort of $\alpha$, starting at a point $z$, then $|z|<|y|$, because the
  restriction of the curve from $y$ to $z$ is a linking curve.
  
  Further, $\alpha_{0}$ is unbeatable, so that any non-trivial retort of this
  short curve will increase the radius at most $|x|-|y|- \delta$ for some
  $\delta >0$. The inequality still holds with $\delta /2$ if instead of
  $\alpha_{0}$ we have a GACDC starting at some $x$ in a small enough
  neighborhood $V$ of $x_{0}$.
  
  So if we take $S$ as the intersection of $V$ and a ball of radius $\delta
  /2$, then $( S,O )$ is transient, and the gain is at least $\delta/2$.
  
  Any GACDC contained in $O$ is the graph over any ACDC of a Lipschitz function with derivative bounded by $c$, so it has finite length.
  It follows that the pair is bounded.
  
  \item[$A_{3}$] We recall that the set of singular points $\mathcal{C}$ near
  an $A_{3}$ point is an hypersurface, and the stratum of $A_{3}$ points is a
  smooth curve. An ACDC starting at any $A_{2}$ point will flow either into
  the stratum of $A_{3}$ points transversally (within $\mathcal{C}$), or into
  the boundary of $O$.
  
  For points in a smaller neighborhood $V \subset O$, one of the following
  things happen:
  \begin{itemize}
    \item If an ACDC starting at $x \in V \cap \mathcal{A}_{2}$ flows into an
    $A_{3}$ point, then it can be replied in one step, and the algorithm
    stops. The algorithm also stops if $x \in \mathcal{A}_{3}$.
    
    \item If the ACDC starting at $x \in V \cap \mathcal{A}_{2}$ flows into $y
    \in \partial O$, the argument is the same as that for an $A_{2}$ point.
  \end{itemize}
  The length of any GACDC in $O$ is bounded for the same reason as for $A_2$ points, and this is enough to bound aspirant curves contained in $O$.
  \item[$A_{4}$] Near an $A_{4}$ point, $\mathcal{C}$ is a smooth hypersurface
  and $\mathcal{A}_{3}$ is a smooth curve sitting inside $\mathcal{C}$. The
  $A_{4}$ point is isolated and splits the curve $\mathcal{A}_{3}$ into two
  parts. One of them, which we call Branch I, consists of $A_{3} ( I )$
  points, and the other branch consists of $A_3(II)$ points. The conjugate
  distribution $D$ coincides with the kernel of $\e$ at the $A_{4}$ point, and
  is contained in the tangent to the manifold of $A_{3}$ points.
  
  As we saw before, a CDC that ends up in the $A_{4}$ point can be perturbed
  so that it either hits an $A_{3}$ point, or leaves the neighborhood.
  
  Let $H$ be the set of points such that the CDC starting at that point flows
  into the $A_{4}$ point. $H$ is a smooth curve, and splits $U$ into two
  parts. One of them, $U_{1}$, contains only $A_{2}$ points, while the other,
  $U_{2}$, contains all the $A_{3}$ points.
  
  \begin{figure}[ht]
    \includegraphics[width=0.8\textwidth]{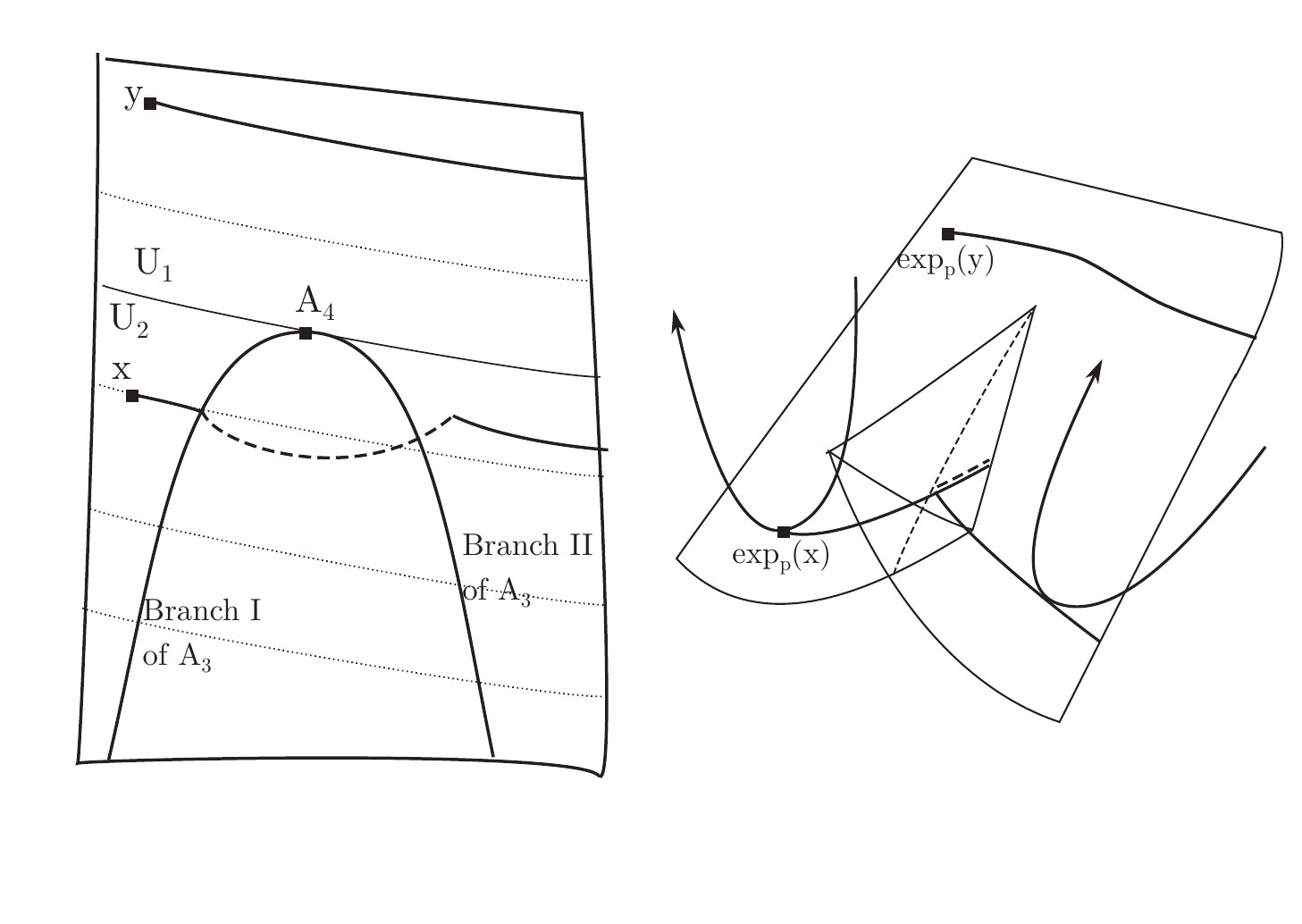}
    \caption{\label{figure: linking curves near A4 points}This picture shows a
    neighborhood of an $A_{4}$ point in $\T$, together with the linking curves
    that start at $x$ and $y$ (to the left) and the image of the whole sketch
    by $\e$ (to the right).}
  \end{figure}
  
  Look at figure \ref{figure: linking curves near A4 points}: a CDC starting
  at a point $y \in U_{1}$ flows into the boundary of $U$ without meeting any
  obstacle. A CDC $\alpha$ starting at a point $x \in U_{2}$, however, flows
  into the branch I of $\mathcal{A}_{3}$. We can start a retort $\beta$ at
  that point, but it will get interrupted when $\exp \circ \beta$ reaches the
  stratum of the queue d'aronde that is the image of two strata of $A_{2}$
  points meeting transversally. The retort cannot go any further because only
  the points ``above $\exp ( \mathcal{C} )$'' (the side of ) have a preimage,
  and points in the main sheet of $\exp ( \mathcal{C} )$ have only one
  preimage, that is $A_{2}$. When he hit the stratum of $A_{2}$ points, we
  follow a CDC to get a curve that leaves the neighborhood in a similar way as
  the curve starting at $y$ did.
  
  \item[$D_{4}$] Any CDC starting at any point in a neighborhood of a
  $D^{-}_{4}$, or $D_{4}^{+}$ of type I point leaves the neighborhood without
  meeting other singularities. A nearby GACDC will also do. We only have to
  worry about the one CDC that flows into the $D_{4}^{+}$ of type II, but we
  always take a nearby GACDC that avoids the center.
\end{description}

\section{Further questions}\label{section: beyond}

We have proposed a new strategy for proving the Ambrose conjecture.
If our only goal had been to prove that the Ambrose conjecture holds for a generic family of metrics, we could have simplified the definitions of unequivocal point and linked points.
We have chosen the definitions so that the sutured property does not exclude some common manifolds.

There is a weaker form of the sutured property that may be simpler to prove, allowing for a remainder set K consisting of points that are neither unequivocal nor linked to an unequivocal point, but such that the Hausdorff dimension of $e_1(\mathcal{K})$ is smaller than $n-2$.
We have decided not to include it here, but the reader can find details in chapter 6.5 of \cite{Angulo Tesis}.

\subsection{Bounding the length of the linking curves}

It doesn't seem likely that a uniform bound can be found for the lengths of the FCLCs built with the linking curve algorithm.
Let us show how a naive argument for bounding the length fails at giving a uniform bound.

Let $B_R$ be the maximum length of a linking curve starting at a point $x$ of radius $R$.
The algorithm starting at $x$ first adds a GACDC $\alpha$ of length $l$ that leaves a transient neighborhood $U$ of $x$.
The next iterations of the algorithm add a linking curve $\beta$ at the tip of $\alpha$, and it only remains to reply to $\alpha$.
If this could be done in one step, we would have:
$$
B_R < 2l + B_{R-\varepsilon}
$$
but unfortunately, $\alpha =  \alpha_{1} \ast \ldots \ast \alpha_{k}$ might cross $\SAtwo$ $k$ times. After adding $\beta$ to the tip of $\alpha$, we can always reply to $\alpha_k$, but then we may have to iterate the algorithm until we add a linking curve starting at the tip of $\alpha_{k-1}$ before we can reply to the $\alpha_{k-1}$.
This means we may have to add $k$ linking curves, and our bound is only:
$$
B_R < 2l + k\cdot B_{R-\varepsilon}
$$
This is of little use unless we can bound $k$.

However, it may be enough to find a uniform bound of the composition of the linking curve with $e_1$.
Then a metric can be approximated by generic ones, obtaining sequences of linking curves for the approximate metrics, and then using \cite[Lemma 4.2]{Hambly Lyons}, for instance.

\end{document}